\documentclass{amsart}
\usepackage{amssymb}
\usepackage{amsfonts}
\usepackage{mathrsfs}
\usepackage{bbm}
\usepackage[dvipdfmx, dvipsnames, svgnames, x11names]{xcolor} %
\theoremstyle{plain}
    \newtheorem{theorem}{Theorem}[section]
    \newtheorem{lemma}[theorem]{Lemma}
    \newtheorem{proposition}[theorem]{Proposition}

\theoremstyle{definition}
    
    \newtheorem{remark}{Remark}[section]
    \newtheorem*{acknowledgement}{Acknowledgement}
\theoremstyle{remark}
    
\numberwithin{equation}{section}

\renewcommand{\l}{\left}
\renewcommand{\r}{\right}
\newcommand{\cleq}{\lesssim}

\newcommand{\ceq}{\approx} %\sim, \eqsim, \simeq, or \approx
\def\norm#1{\left\Vert #1 \right\Vert} %norm
\def\jbra#1{\left\langle #1 \right\rangle} %japanese bracket
\newcommand{\R}{\mathbb{R}}

\newcommand{\Z}{\mathbb{Z}}

\newcommand{\cD}{\mathcal{D}}

\newcommand{\cF}{\mathcal{F}}

\let\div\relax
\DeclareMathOperator{\div}{div}
\begin{document}
%%%%%%%%%%%%%%%%%%%%%%%%%%%%%%%%%%%%%%%%%%%%%%%%%%%%%%%%

\title[Non-delay limit from NLDW to NLH]{Non-delay limit in the energy space from the nonlinear damped wave equation to the nonlinear heat equation}

\author[T. Inui]{Takahisa Inui}
\address{Department of Mathematics, Graduate School of Science, Osaka University, Toyonaka, Osaka 560-0043, Japan}
\email{inui@math.sci.osaka-u.ac.jp}

\author[S. Machihara]{Shuji Machihara}
\address{Department of Mathematics, Faculty of Science, Saitama University, 255 Shimo-Okubo, Sakura-ku, Saitama City 338-8570, Japan}
\email{machihara@rimath.saitama-u.ac.jp}
\date{}
\maketitle

\begin{abstract}
%The heat equation is derived from the Fourier law. 
%The damped wave equation is derived from the Cattaneo law instead of the Fourier law which derives the heat equation. The Cattaneo law is the Fourier law with a time delay parameter $\tau$. 
%We consider the non-delay limit problem for those equations with a power type nonlinearity. 
We consider a singular limit problem from the damped wave equation with a power type nonlinearity to the corresponding heat equation. 
%Formally, the solution %of the damped wave equation 
%goes to that of the heat equation as time delay parameter $\tau \to 0+$. 
We call our singular limit problem  non-delay limit. % in the present paper. 
%For these linear equations, this phenomena was investigated e.g. by \cite{Fat85}. 
Our proofs are based on the argument for non-relativistic limit from the nonlinear Klein-Gordon equation to the nonlinear Schr\"{o}dinger equation by the second author, Nakanishi, and Ozawa \cite{MNO02}, Nakanishi \cite{Nak02}, and Masmoudi and Nakanishi \cite{MaNa02}. %Unlike those equations, w
We can obtain better results for the non-delay limit problem than that for the non-relativistic limit problem due to the dissipation property. More precisely, we get the better convergence rate of the $L^2$-norm and we also obtain the global-in-time uniform convergence of the non-delay limit in the $L^2$-supercritical case. 
\end{abstract}

\tableofcontents

%%%%%%%%%%%%%%%%%%%%%%%%%%%%%%%%%%%%%%%%%%%%%%%%%%%%%%%%

%%%%%%%%%%%%%%%%%%%%%%%%%%%%%%%%%%%%%%%%%%%%%%%%%%%%%%%%

\section{Introduction}

\subsection{Background}
We consider the following nonlinear damped wave equation.
\begin{align}
\label{NLDW}
\tag{NLDW$_{\tau}$}
\begin{cases}
\tau \partial_t^2 u_{\tau} -\Delta u_{\tau} +\partial_t u_{\tau} = \mu |u_{\tau}|^{p-1}u_{\tau}, & (t,x) \in (0,T) \times \R^d,
\\
(u_\tau(0), \partial_t u_{\tau}(0)) =(f_{\tau}, g_{\tau}), & x \in  \R^d,
\end{cases}
\end{align}
where $\tau >0$, $T>0$, $d \in \mathbb{N}$, $1<p<1+4/(d-2)$ if $d \geq 3$ and $1<p<\infty$ if $d =1, 2$, $\mu = \pm 1$, and $(f_{\tau},g_{\tau}) \in H^1(\mathbb{R}^d) \times L^2(\mathbb{R}^d)$ are given initial data depending on $\tau$. The nonlinerity is called focusing when $\mu=1$ and defocusing when $\mu=-1$. The power of the nonlinearity is called energy subcritical. 
It is well known that the energy $E_{\tau}$ decays, where it is defined by 
\begin{align*}
	E_{\tau}(u_{\tau}(t)):= \frac{1}{2} \| \nabla u_{\tau}(t)\|_{L^2} +  \frac{\tau}{2} \| \partial_t u_{\tau}(t)\|_{L^2} - \frac{\mu}{p+1} \|u_{\tau}(t)\|_{L^{p+1}}.
\end{align*}
Indeed, we have
\begin{align*}
	\frac{d}{dt} E_{\tau}(u_{\tau}(t)) = -\|u_{\tau}(t)\|_{L^2} \leq 0,
\end{align*}
where $u_{\tau}$ is the solution of \eqref{NLDW}. Thus, the global solution of \eqref{NLDW} decays. 

The parameter $\tau$ denotes time delay. To explain this, we recall the derivation in the current literature of the linear damped wave equation %.The linear damped wave equation
\begin{align}
\label{DW}
\tag{DW$_\tau$}
	\tau \partial_t^2 \phi_{\tau} -\Delta \phi_{\tau} +\partial_t \phi_{\tau}=0.
\end{align}
by the Cattaneo law (see e.g. \cite[(1.9)--(1.11)]{Str11}). It is well known that the heat equation $\partial_t \phi - \Delta \phi =0$ is derived from the Fourier law. More precisely, letting $\phi$ denote temperature and $q$ denote heat flux, then we have
\begin{align}
\label{eq1.1}
	\partial_t \phi = - \div q.
\end{align}
The Fourier law implies that the flux depends linearly on the derivative of temperature $\phi$, i.e.
\begin{align}
\label{eq1.2}
	q = - \nabla \phi,
\end{align}
where we take thermal conductivity as 1.
Combining \eqref{eq1.1} and \eqref{eq1.2}, we obtain the heat equation. On the other hand, the Cattaneo law implies the damped wave equation. The Cattaneo law states that the flux $q$ does not depend linearly on $\nabla \phi(t)$ at the same time, but it depends linearly on $\nabla \phi (t-\tau)$ with a slight time lag $\tau$. We replace \eqref{eq1.2} by 
\begin{align}
\label{eq1.3}
	q(t,x) = - \nabla \phi ( t-\tau,x),
\end{align}
where $\tau>0$. From this, we have
\begin{align*}
	q(t+\tau,x) = - \nabla \phi ( t,x).
\end{align*}
By the Taylor expansion of the left hand side at $\tau=0$, we obtain
\begin{align*}
	q(t+\tau,x) = \sum_{n=0}^{\infty} \frac{ \partial_t^{(n)} q(t,x) }{n!} \tau^n.
\end{align*}
Since $\tau$ is small, we ignore the higher terms ($n\geq 2$) and thus we obtain 
\begin{align*}
	q(t,x) + \tau \partial_t q(t,x) = - \nabla \phi ( t,x).
\end{align*}
%\eqref{eq1.1} implies that
%\begin{align*}
%	\partial_t^2 \phi =- \div \partial_t q
%\end{align*}
Combining this with \eqref{eq1.1}, we obtain the damped wave equation
\begin{align*}
	\tau \partial_t^2 \phi + \partial_t \phi = -\tau \div \partial_t q - \div q= \div(\nabla \phi) = \Delta \phi
\end{align*}
and $\tau$ denotes the time delay effect, which is called a relaxation time and given by the inverse square of speed of the second sound. 
If $\tau$ goes to 0, then the damped wave equation formally goes to the heat equation. 
In the present paper, the singular limit $\tau \to 0$ is called non-delay limit. %The non-delay limit problem for the linear 
The non-delay limit problem for the linear equations is mathematically investigated well (see e.g. \cite{Fat83,Fat85}). In the present paper, we consider the non-delay limit problem for the nonlinear equation \eqref{NLDW}. 
As $\tau \to 0$, we formally  obtain the nonlinear heat equation
\begin{align*}
	\partial_t v - \Delta v = \mu |v|^{p-1}v.
\end{align*}
We will show mathematically that the solution of \eqref{NLDW} converges to the solution of the nonlinear heat equation. This problem was partially studied in \cite{Naj93,Naj98}. 
%We treat all power $p$ less than the energy critical power. 
On the other hand, singular limit problems for nonlinear relativistic equation to nonlinear dispersive equation are investigated. The second author, Nakanishi, and Ozawa \cite{MNO02} shows that the solution of the energy subcritical nonlinear Klein--Gordon equation goes to that of the energy subcritical nonlinear Schr\"{o}dinger equation if the speed of the light goes to infinity. 
More precisely, they consider the nonlinear Klein--Gordon equation
\begin{align*}
	\tau \partial_t^2 w_{\tau} - \Delta w_{\tau}+ \tau^{-1} w_{\tau} =\mu |w_{\tau}|^{p-1}w_{\tau}
\end{align*}
where $\tau^{-1/2}=c$ denotes the speed of the light. By modulated function $\psi_{\tau}=e^{-i\tau^{-1}t}w_{\tau}$, they transform the nonlinear Klein--Gordon equation into 
\begin{align*}
	\tau \partial_t^2 \psi_{\tau} - \Delta \psi_{\tau}+ i \partial_t \psi_{\tau} =\mu |\psi_{\tau}|^{p-1}\psi_{\tau}.
\end{align*}
They showed that the solution $\psi_{\tau}$ goes to the solution $\psi_{0}$ of the following nonlinear Schr\"{o}dinger equation 
\begin{align}
\tag{NLS}
\label{NLS}
	i \partial_t \psi_0 - \Delta \psi_{0} = \lambda |\psi_{0}|^{p-1}\psi_{0}
\end{align}
as $\tau$ goes to $0$, where this means the speed of the light $c$ goes to infinity and it is called non-relativistic limit. More concretely, they showed local-in-time uniform convergence, i.e.
\begin{align*}
	\| \psi_{\tau} -\psi_{0} \|_{L^{\infty}(0,T: H^1)} \to 0 \text{ as } \tau \to 0
\end{align*}
for finite fixed $T$ less than the maximal existence time of the solution to \eqref{NLS}. 
%This limit is called non-relativistic limit. 
See \cite{MNO03,MaNa03,FLS16} for other equations. 

We will give such a statement for the non-delay limit from \eqref{NLDW} to the nonlinear heat equation. 
We apply the arguments in \cite{MNO02} and the sequel works by Nakanishi \cite{Nak02} and Masmoudi and Nakanishi \cite{MaNa02} to our non-delay limit problem. 

However, it is not mere application and we obtain better results than them due to the dissipation. 
More precisely, we obtain the following difference between our equations and their equations. 
The almost optimal convergence rate of the $L^2$-norm is $\tau^{1/4}$ in the non-relativistic limit problem. However, we will find that the rate is $\tau^{1/2}$ for non-delay limit. 
Moreover,  we can obtain global-in-time uniform convergence for the non-delay limit in the $L^2$-supercritical case, i.e. $p>1+4/d$, though this does not hold for the non-relativistic limit as pointed out in \cite{Nak02}. 
%In \cite{Nak02}, he showed that the wave, inverse wave, and scattering operators converge in the $L^2$-supercritical case, i.e. $p>1+4/d$. 
%we can obtain global-in-time uniform convergence for the non-delay limit in the $L^2$-supercritical case, i.e. $p>1+4/d$. 

%%%%%%%%%%%%%%%%%%%%%%%%%%%%%%%%%%%%%%%%%%%%%%%%%%%%%%%%
\subsection{Main result}

%Let $u_{\tau}$ be a solution of \eqref{NLDW}. 
It is well known that there exist $T_{\tau}=T(\tau, \|f_\tau\|_{H^1}, \tau^{1/2}\|g_{\tau}\|_{L^2})>0$  and the solution $u_{\tau} \in C([0,T_{\tau}): H^1(\mathbb{R}^d))$ to \eqref{NLDW} for fixed $\tau$.% if $(f_{\tau},g_{\tau}) \in H^1(\mathbb{R}^d) \times L^2(\mathbb{R}^d)$. 

First, we have uniform boundedness of the solution $u_{\tau}$. 

\begin{theorem}[Uniform bound]
\label{thm1.1}
Let $u_{\tau} \in C([0,T_{\tau}): H^1(\mathbb{R}^d))$ be a solution to \eqref{NLDW} and $T_{\tau}^{*}$ be the maximal existence time. If the initial data $(f_{\tau},g_{\tau}) \in  H^1(\mathbb{R}^d) \times  L^2(\mathbb{R}^d)$ satisfies
\begin{align*}
	\limsup_{\tau \to 0} ( \| f_{\tau} \|_{H^1} + \tau^{\frac{1}{2}} \|g_{\tau}\|_{L^2}) < \infty,
\end{align*}
then we have
\begin{align*}
	T^{*}:=\liminf_{\tau \to 0} T_{\tau}^{*}>0
\end{align*}
and $u_{\tau}$ satisfies that for any $T<T^{*}$ there exist a constant $C_T>0$ and $\tau_T>0$ such that 
\begin{align*}
	\|u_{\tau}\|_{L^{\infty}(0,T:H^1)} 
	+ \tau^{\frac{1}{2}}\| \partial_t u_{\tau} \|_{L_t^{\infty}(0,T:L^2)}
	\leq C_T
\end{align*}
for any $\tau \in (0,\tau_T)$.
\end{theorem}

%\begin{remark}
%As seen in Theorem \ref{thm} below, if $f_{\tau}$ converges in $H^1$, then $T^{*}$ is less than the maximal existence time of the solution to \eqref{NLH}. 
%\end{remark}

Let $v$ be a solution of 
\begin{align}
\label{NLH}
\tag{NLH}
	\begin{cases}
	\partial_t v - \Delta v = \mu |v|^{p-1}v, & (t,x) \in (0,T) \times \R^d,
	\\
	v(0) =f, & x \in \mathbb{R}^d
	\end{cases}
\end{align}
It is well known that local well-posedness of \eqref{NLH} in $H^1$ holds.

Next, we have the following $L^2$-convergence result in the non-delay limit.

\begin{theorem}[$L^2$-convergence with rate]
\label{thm1.2}
Let the assumptions of Theorem \ref{thm1.1} be satisfied. 
Let $f\in H^1(\mathbb{R}^d)$, $v$ be a solution to \eqref{NLH}, and $T_{\max}(v)$ is the maximal existence time of the solution $v$ to \eqref{NLH}. 
%Let $u_{\tau} \in C([0,T_{\tau}): H^1(\mathbb{R}^d))$ to \eqref{NLDW} and $T^{*}$ be as in Theorem \ref{thm1.1}. If the initial data $(f_{\tau},g_{\tau}) \in  H^1(\mathbb{R}^d) \times  L^2(\mathbb{R}^d)$ satisfies
%\begin{align*}
%	\limsup_{\tau \to 0} ( \| f_{\tau} \|_{H^1} + \tau^{\frac{1}{2}} \|g_{\tau}\|_{L^2}) < \infty,
%\end{align*}
%and 
%If we have
%\begin{align*}
%	f_{\tau} \to f \text{ in } L^2(\mathbb{R}^d) \text{ as } \tau \to 0,
%\end{align*}
Then, for any $T<\min\{T^{*},T_{\max}(v)\}$, we have
\begin{align*}
	\| u_{\tau} - v\|_{L^{\infty}(0,T: L^2(\mathbb{R}^d))} 
	\lesssim   \|f_{\tau}-f\|_{L^2} + \tau \|g_{\tau}\|_{L^2}  + \tau^{\frac{1}{2}}.
\end{align*}
for any $\tau \in (0,\tau_T)$, where the implicit constant is independent of $\tau$. 
\end{theorem}

\begin{remark}
For the non-relativistic problem, the optimal rate of $L^2$-convergence was reported in \cite{MNO02}. That is $\tau^{1/4}$. On the other hand, we obtain $\tau^{1/2}$ in the non-delay limit problem. This is a main difference between dispersive equations and dissipative equation. 
To explain roughly this difference, we consider the main parts of the convergence. 
For the non-relativistic problem, we have
\begin{align*}
	\| (e^{it\frac{1-\sqrt{1+4\tau|\xi|^2}}{2\tau}} - e^{it|\xi|^2}) \hat{f} \|_{L^2(|\xi|\leq \tau^{-1/4})}
	&\lesssim \int_{0}^{\tau} \left\|  i t |\xi|^4 e^{it\frac{1-\sqrt{1+4\tilde{\tau}|\xi|^2}}{2\tilde{\tau}}} \hat{f} \right\|_{L^2(|\xi|\leq \tau^{-1/4})}d\tilde{\tau} 
	\\
	&\lesssim  \tau^{\frac{1}{4}} t\|f\|_{H^1}
\end{align*} 
and we have
\begin{align*}
	\| (e^{it\frac{1-\sqrt{1+4\tau|\xi|^2}}{2\tau}} - e^{it|\xi|^2}) \hat{f} \|_{L^2(|\xi|> \tau^{-1/4})}
	\lesssim \tau^{\frac{1}{4}} \|f\|_{H^1}
\end{align*}
since $|\xi|^{-1}\leq \tau^{1/4}$ in the high frequency region. 
This estimate implies that the optimal convergence rate is $\tau^{1/4}$. 
On the other hand, for our problem, we have 
\begin{align*}
	\| (e^{t\frac{1-\sqrt{1-4\tau|\xi|^2}}{2\tau}} - e^{-t|\xi|^2} )\hat{f} \|_{L^2(|\xi|\leq \tau^{-1/2})}
	&\lesssim \int_{0}^{\tau} \left\|  t |\xi|^4 e^{t\frac{1-\sqrt{1-4\tilde{\tau}|\xi|^2}}{2\tilde{\tau}}} \hat{f}  \right\| _{L^2(|\xi|\leq \tau^{-1/2})}d\tilde{\tau}
	\\
	&\lesssim  \tau^{\frac{1}{2}} \|f\|_{H^1}.
\end{align*}
since $t|\xi|^2 e^{t\frac{1-\sqrt{1+4\tilde{\tau}|\xi|^2}}{2\tilde{\tau}}} \lesssim   t|\xi|^2e^{ -t|\xi|^2}\lesssim  1$ and we also have
\begin{align*}
	\| (e^{t\frac{1-\sqrt{1-4\tau|\xi|^2}}{2\tau}} - e^{-t|\xi|^2} )\hat{f} \|_{L^2(|\xi|> \tau^{-1/2})}
	\lesssim \tau^{\frac{1}{2}}\|f\|_{H^1}.
\end{align*}
This shows that our rate is $\tau^{1/2}$. 

%
%For the non-relativistic problem, $e^{it\frac{}{2\tau}}$
%
%As in \cite{MNO02}, we get the following optimal convergence rate in $L^2$ for the non-relativistic problem. Indeed, we can obtain 
%\begin{align*}
%	\| u_{\tau} - v\|_{L^{\infty}(0,T: L^2(\mathbb{R}^d))} 
%	\lesssim   \|f_{\tau}-f\|_{L^2} + \tau \|g_{\tau}\|_{L^2}  + o(\tau^{\frac{1}{4}}).
%\end{align*}
%%in the same way as in \cite{MNO02}. 
%%However, we do not pursue this in the paper. 
\end{remark}

We have the following locally uniform $H^1$-convergence of the non-delay limit. 

\begin{theorem}[$H^1$-convergence]
\label{thm1.3}
%Let the assumption of Theorem \ref{thm1.1} be satisfied. 
Let $f\in H^1(\mathbb{R}^d)$, $v$ be a solution to \eqref{NLH}, and $T_{\max}(v)$ is the maximal existence time of the solution $v$ to \eqref{NLH}. 
If the initial data $(f_{\tau},g_{\tau}) \in  H^1(\mathbb{R}^d) \times  L^2(\mathbb{R}^d)$ satisfies
\begin{align*}
	(f_{\tau},\tau^{\frac{1}{2}} g_{\tau}) \to (f,0) \text{ in } H^1(\mathbb{R}^d)\times L^2(\mathbb{R}^d) \text{ as } \tau \to 0,
\end{align*} 
then we have $T^{*}\geq T_{\max}(v)$ and we have
\begin{align*}
	\| u_{\tau} - v\|_{L^{\infty}(0,T: H^1(\mathbb{R}^d))} 
	+ \tau^{\frac{1}{2}} \| \partial_t u_{\tau}\|_{L^{\infty}(0,T: L^2(\mathbb{R}^d))}
	\to 0
\end{align*}
as $\tau \to 0$ for any $T<T_{\max}(v)$. 
\end{theorem}

We remark that the assumptions in Theorem \ref{thm1.1} are satisfied under the assumptions of Theorem \ref{thm1.3}.

Due to the dissipation, we obtain the following global convergence result in the $L^2$-supercritical case. %, i.e. $p>1+4/d$. %This proof is based on the convergence of the inverse wave operator for nonlinear Klein-Gordon equation to nonlinear Schr\"{o}dinger equation by Nakanishi \cite{Nak02}. 

\begin{theorem}[Global $H^1$-convergence]
\label{thm1.4}
Let $1+4/d < p < 1+4/(d-2)$. Assume that the solution $v$ to \eqref{NLH} is global and $\| v(t) \|_{H^1}$ decays to $0$ as $t \to \infty$. If the initial data $(f_{\tau},g_{\tau}) \in  H^1(\mathbb{R}^d) \times  L^2(\mathbb{R}^d)$ satisfies
\begin{align*}
	(f_{\tau},\tau^{\frac{1}{2}} g_{\tau}) \to (f,0) \text{ in } H^1(\mathbb{R}^d)\times L^2(\mathbb{R}^d) \text{ as } \tau \to 0.
\end{align*}
Then we have
\begin{align*}
	\lim_{\tau \to 0}\left( \|u_{\tau} - v\|_{L^{\infty}(0,\infty: H^1(\mathbb{R}^d))}+ \tau^{\frac{1}{2}} \| \partial_t u_{\tau}\|_{L^{\infty}(0,\infty: L^2(\mathbb{R}^d))}\right) = 0.
\end{align*}
%
%there exists $\delta>0$ (independent of $\tau$) such that if $\| f_{\tau}\|_{H^1} + \tau^{\frac{1}{2}} \|g_{\tau}\|_{L^2}<\delta$ then $u_{\tau}$ is global and the following uniform bound holds. 
%\begin{align*}
%	\sup_{\tau>0}\| u_{\tau} \|_{L^\infty(0,\infty;H^1)} \lesssim \delta.
%\end{align*}
%Moreover, if the assumption fulfilled, then we have
%\begin{align*}
%	\lim_{t \to \infty} \| u_{\tau}(t)\|_{\dot{H}^1} = 0
%\end{align*}
%uniformly in $\tau$, that is, for any $\varepsilon>0$ there exists $T$ independent of $\tau$ such that
%$\|\nabla u_{\tau}(t)\|_{L^2} <\varepsilon$ for all $t>T$. 
\end{theorem}

\begin{remark}
The non-relativistic problem from NLKG to NLS, the global convergence does not hold (see \cite{Nak02}). 
Theorem \ref{thm1.4} depends essentially on the dissipation property. 
\end{remark}

The following theorem relies on the fact that the spatial derivative implies the additional time decay $t^{-1/2}$ for the damped wave equation and the heat equation.

\begin{theorem}[$\dot{H}^1$-decay order]
\label{cor1.5}
If the assumption in Theorem \ref{thm1.4} is satisfied, then we have
\begin{align*}
	 t^{\frac{1}{2}} \| u_{\tau}(t)\|_{\dot{H}^1} \to 0 \text{ as } t \to \infty
\end{align*}
uniformly in $\tau$. Especially, we obtain
\begin{align*}
	\lim_{\tau \to 0} \left\| t^{\frac{1}{2}}( u_{\tau} - v) \right\|_{L^{\infty}(0,\infty:\dot{H}^1) }= 0.
\end{align*}
\end{theorem}

%\begin{remark}
%Under the assumption of Theorem \ref{thm1.4}, we find that $ \| u_{\tau}(t)\|_{L^2}$ decays to $0$ as $t \to \infty$ for any fixed $\tau$. However, 
%it is not clear that $\lim_{t \to \infty} \| u_{\tau}(t)\|_{L^2} = 0$ holds uniformly in $\tau$. Thus, the $L^2$-convergence on global time interval, namely $\lim_{\tau \to 0+}\| u_{\tau} - v\|_{L^{\infty}(0,\infty: L^2(\mathbb{R}^d))} = 0$, is also open. 
%\end{remark}

%%%%%%%%%%%%%%%%%%%%%%%%%%%%%%%%%%%%%%%%%%%%%%%%%%%%%%%%
\subsection{Plan of proofs}

We apply the argument in \cite{MNO02}. 
To show Theorem \ref{thm1.3}, we prepare the uniform boundedness in $\tau$, Theorem \ref{thm1.1}. The uniform boundedness in the case of $d=1,2$ can be shown easily. Indeed, it is enough to use the space $L_t^\infty H_x^1$ in order to obtain the closed estimate. For the high dimensional case $d \geq 3$, we need  the Strichartz norms and to divide them into the low and high frequency parts. The low frequency of the solution of the damped wave equation behaves like that of the heat equation and the high frequency part of that behaves like that of the wave equation. 
However, it is not trivial that the low (resp. high) frequency part of the nonlinearity is the nonlinearity of the low (resp. high) frequency part. To overcome this difficulty, we use argument for the non-relativistic limit. The second author, Nakanishi, and Ozawa \cite{MNO02} shows the nonlinear estimate of the frequency decoupling (see \cite[Lemma 3.4]{MNO02}). By applying this lemma, we obtain the closed estimate and thus the uniform boundedness. %We also apply their lemma to obtain the uniform boundedness. %We note that the derivative loss is different between the Klein-Gordon equation, which is considered in \cite{MNO02}, and the damped wave equation. Thus, it is not clear that the same function spaces as in \cite{MNO02} are applicable to our problem. Fortunately, via our calculations, we find that they are applicable. 
Next, we will show the $L^2$-convergence by a direct calculation. By using the dissipation, we obtain the better rate of the $L^2$-convergence rate. At last, we will show the $H^1$-convergence by combining $L^2$-convergence with a compactness argument. In this argument, the energy and charge conservation laws were used in \cite{MNO02}. However, the energy decays for the damped wave equation. Thus, we could not apply their argument. We will apply the compactness argument by Masmoudi and Nakanishi \cite{MaNa02}. They brushed Lemma 3.4 in \cite{MNO02} up and prove the $H^1$-convergence by a compactness method without the energy conservation laws. %We will apply their argument to show the $H^1$-convergence. 

By using the argument of Nakanishi \cite{Nak02}, we obtain Theorem \ref{thm1.4}. He showed the convergence of the wave, inverse wave, and scattering operators in the non-relativistic problem. %However, this does not imply the global convergence. For our dissipative equations, the solutions will decay to $0$. Thus, 
By applying his argument and the dissipation, we obtain the global convergence unlike the non-relativistic problem. 
Theorem \ref{cor1.5} relies on the heat-like property that the spatial first-order derivative implies the additional time decay $t^{-1/2}$. %This explicit decay gives us the uniform time decay in $\tau$. 

%\emph{Construction of the paper. }

%%%%%%%%%%%%%%%%%%%%%%%%%%%%%%%%%%%%%%%%%%%%%%%%%%%%%%%%

%\emph{Notation.}

%The exponent $\alpha$ is the exponent appearing from the scaling and $\gamma$ and $\delta$ are the derivative losses appearing in the Strichartz estimates.

%%%%%%%%%%%%%%%%%%%%%%%%%%%%%%%%%%%%%%%%%%%%%%%%%%%%%%%%
%\section{Proofs}
%%%%%%%%%%%%%%%%%%%%%%%%%%%%%%%%%%%%%%%%%%%%%%%%%%%%%%%%
\section{Preliminaries}
%%%%%%%%%%%%%%%%%%%%%%%%%%%%%%%%%%%%%%%%%%%%%%%%%%%%%%%%
\subsection{Notation}
For an exponent $p \in [1,\infty]$, the H\"{o}lder conjugate of $p$ is denoted by $p'$.

Let $\chi_{\leq 1} \in C_0^{\infty}(\mathbb{R}^d)$ be a radially symmetric cut-off function such that $\chi_{\leq 1}(\xi) =1$ if $|\xi| \leq 1$ and $\chi_{\leq 1}(\xi) =0$ if $|\xi| \geq 2$. We set $\chi_{\leq a}(\xi) := \chi_{\leq 1}(\xi/a)$ for a positive number $a$. We set $\chi_{> a}:= 1-\chi_{\leq a}$.

We denote the Fourier transform and its inverse by $\mathcal{F}$ and $\mathcal{F}^{-1}$, respectively. % the Fourier transform and its inverse. 
We also denote the Fourier transform of a function $f$ by $\hat{f}$. 
For a measurable function $m$, we define the Fourier multiplier $m(\nabla)$ by $m(\nabla):=\mathcal{F}^{-1}m(\xi) \mathcal{F}$. Thus, $\chi_{\leq \tau^{-1/2}}(\nabla)=\mathcal{F}\chi_{\leq \tau^{-1/2}}(\xi) \mathcal{F}$. We sometimes omit $\nabla$, that is, we use $\chi_{\leq \tau^{-1/2}}$ instead of $\chi_{\leq \tau^{-1/2}}(\nabla)$. We denote $P_j:=\mathcal{F} (\chi_{\leq 2^j} (\xi) -\chi_{\leq 2^{j-1}}(\xi)  )\mathcal{F}$ for $j \in \mathbb{Z}$. We also write $P_j$ to denote the symbol of $P_j$.  We set $\langle a \rangle:=(1+a^2)^{1/2}$. 

We use $A \lesssim B$ to denote the estimate $A \leq CB$ with some constant $C>0$.
The notation $A \approx B$ stands for $A \lesssim B$ and $B \lesssim A$.

For a time interval $I$, we set the space-time function space $L_{t}^{q}L_{x}^{r}(I)$ by $L_{t}^{q}(I; L_{x}^{r})$ whose norm is 
\begin{align*}
	\| f \|_{L_{t}^{q}L_{x}^{r}(I)} := \left\{ \int_{I} \left( \int_{\mathbb{R}^d} |f(t,x)|^r dx \right)^{\frac{q}{r}} dt \right\}^{\frac{1}{q}}.
\end{align*}
We also may omit the time interval $I$. We define the  Sobolev space by 
\begin{align*}
	W^{\sigma,p}(\mathbb{R}^d) := \left\{ f \in \mathscr{S}' (\mathbb{R}^d) :
		\| f \|_{W^{\sigma,p}} = \| \langle \nabla \rangle^\sigma f \|_{L^p} < \infty \right\}
\end{align*}
for $\sigma \geq 0$ and $1\leq p \leq \infty$. We set $H^\sigma:=W^{\sigma,2}$. 
For $1\leq p,q \leq \infty$ and $\sigma \in \mathbb{R}$, we define inhomogeneous Besov norm by
\begin{align*}
	\| f \|_{B_{p,q}^{\sigma}} 
	:=\| \chi_{\leq 1} f \|_{L^p} + \left\|  \{ 2^{j\sigma} \|P_j f \|_{L^p}  \}_{j=1}^{\infty} \right\|_{l^{q}}
\end{align*}
and inhomogeneous Besov space by
\begin{align*}
	B_{p,q}^{\sigma}(\R^d) := \{ f \in \mathscr{S}' : \| f \|_{B_{p,q}^{\sigma}} < \infty \}.
\end{align*}
For a function space $X$, we denote its homogeneous space by $\dot{X}$. For example,  $\dot{X}=L_t^q \dot{B}_{r,p}^{s}$ if $X=L_t^qB_{r,p}^{s}$.

For Banach spaces $X,Y$, we set 
\begin{align*}
	X |Y & := \left\{ f : \| f \|_{X|Y} :=  \| \chi_{\leq \tau^{-1/2} }f\| _{X} + \| \chi_{>\tau^{-1/2} }f \|_{Y} < \infty \right\}  \text{ for } \tau>0,
\end{align*}
and $[X,Y]_{\theta}$ is an complex interpolation space between $X$ and $Y$ of order $\theta \in (0,1)$. We set $\| f \|_{aX} := a^{-1}\|f\|_{X} $. %aX&:=\left\{ f \in X : \| f \|_{aX} := a^{-1}\|f\|_{X} < \infty \right\} \text{ for } a>0, 
%Then, we have  $\| f \|_{[X,Y]_{\theta}} \lesssim \| f \|_{X}^{1-\theta} \| f \|_{Y}^{\theta}$.  

We define 
\begin{align*}
	\mathcal{L}&:=\{\alpha L^q (I; L^r(\mathbb{R}^d)): \alpha>0, 1\leq q \leq \infty,1 < r < \infty \}
	\\
	\mathcal{B}&:=\{\alpha L^q (I; \dot{B}_{r,p}^{s}(\mathbb{R}^d)): \alpha,s>0, 1\leq p,q \leq \infty,1 < r < \infty \}
%	\\
%	\mathcal{B}&:=\{\alpha L^q (I; B_{r,p}^{s}(\mathbb{R}^d)): \alpha,s>0, 1\leq p,q \leq \infty,1 < r < \infty \}
\end{align*}
for an interval $I$. For $Z=\alpha L^q (I; \dot{B}_{r,p}^{s}(\mathbb{R}^d)) \in \mathcal{B}$, we define 
\begin{align*}
	\sigma(Z):=s
\end{align*}
For $X \in \mathcal{L}$ and $Z \in \mathcal{B}$, we set
\begin{align*}
	X^{p-1}Z := \alpha L^{q} (I; \dot{B}_{r,p}^{s}(\mathbb{R}^d))
\end{align*}
where $X = \alpha_0 L^{q_0} (I; L^{q_0}(\mathbb{R}^d))$, $Z=\alpha_1 L^{q_1} (I; \dot{B}_{r_1,p_1}^{s_1}(\mathbb{R}^d))$, and
\begin{align*}
	\frac{1}{q} = \frac{p-1}{q_0} + \frac{1}{q_1},
	\quad
	\frac{1}{r} = \frac{p-1}{r_0} + \frac{1}{r_1},
	\quad 
	p=p_1, \quad  s=s_1, \text{ and } \alpha = \alpha_0^{p-1} \alpha_1.
\end{align*}
This roughly means that $\|  |u|^{p-1}u\|_{X^{p-1}Z} \cleq \|u\|_{X}^{p-1}\|u\|_{Z}$ holds by the fractional Leibnitz rule.%, where $\mathcal{N}(u):=\mu |u|^{p-1}u$. 

For $d\geq 3$, we use $\theta \in (0,1)$ which is the internal division ratio of $p$ between $1$ and $p_1$, i.e. 
\begin{align*}
	p=1-\theta + p_1 \theta,
\end{align*} 
where $p_1:=1+4/(d-2)$. %We have $\theta \in (0,1)$ since $1<p<p_1$. 

For exponents $(q,r)$ and $(\tilde{q},\tilde{r})$ satisfying $q,r,\tilde{q},\tilde{r} \in [ 2,\infty]$, we define exponents $\alpha$ and $\gamma$ by 
\begin{align*}
	\alpha(q,r) &:= -\frac{d}{2}\left(\frac{1}{2}-\frac{1}{r}\right)+\frac{1}{q},
	\\
	\gamma(q,r)&:= \max \left\{ d\left(\frac{1}{2}-\frac{1}{r}\right)-\frac{1}{q}, \frac{d+1}{2}\left(\frac{1}{2}-\frac{1}{r}\right) \right\}
\end{align*}
and $\delta((q,r),(\tilde{q},\tilde{r}))$ in Table \ref{tab1} below.
%%%%%%%%%%%%%%%%%%%%%%%%%%%%%%%%%%%%%%%%%%%%%%%%%%%%%%%%
\subsection{Symbols}

%We consider the linear equation \eqref{DW}. 
The solution propagator of \eqref{DW} is given by  
\begin{align*}
	\begin{pmatrix}
	\phi_{\tau}(t) \\ \partial_t \phi_{\tau} (t)
	\end{pmatrix}
	=  \mathcal{A}_{\tau}(t) 
	\begin{pmatrix}
	f_{\tau} \\  g_{\tau}
	\end{pmatrix}
%	\l( \frac{1}{\tau} \mathcal{D}_{\tau}(t) + \partial_t \mathcal{D}_{\tau}(t) \r) f_{\tau}
%	+ \mathcal{D}_{\tau}(t)  g_{\tau},
\end{align*}
where 
\begin{align*}
	\mathcal{A}_{\tau}(t)& :=
	\begin{pmatrix}
	\frac{1}{\tau} \mathcal{D}_{\tau}(t) + \partial_t \mathcal{D}_{\tau}(t) & \mathcal{D}_{\tau}(t)
	\\
	\frac{1}{\tau} \partial_t \mathcal{D}_{\tau}(t) + \partial_t^2 \mathcal{D}_{\tau}(t) & \partial_t \mathcal{D}_{\tau}(t)
	\end{pmatrix},
	\\
	\widehat{\cD_{\tau}}(t)&:= \frac{1}{ \lambda_{\tau}^{-} -  \lambda_{\tau}^{+}} ( -e^{t \lambda_{\tau}^{+}}+e^{t \lambda_{\tau}^{-}}) \cF,
\end{align*}
and
\begin{align*}
	\lambda_{\tau}^{\pm} = \lambda_{\tau}^{\pm} (\xi) := \frac{-1\pm \sqrt{1-4\tau |\xi|^2}}{2\tau}.
\end{align*}
See e.g. \cite{Mat76}. Therefore, by the Duhamel formula, the solution of the nonlinear equation \eqref{NLDW} is given by
\begin{align*}
	\begin{pmatrix}
	u_{\tau}(t) \\ \partial_t u_{\tau} (t)
	\end{pmatrix}
	= \mathcal{A}_{\tau}(t) 
	\begin{pmatrix}
	f_{\tau} \\  g_{\tau}
	\end{pmatrix}
	+\int_{0}^{t} \frac{1}{\tau}\mathcal{A}_{\tau}(t-s) 
	\begin{pmatrix}
	0 \\  \mathcal{N}(u_{\tau}(s))
	\end{pmatrix}
	ds,
\end{align*}
where we set $ \mathcal{N}(u_{\tau}(s))=\mu |u_{\tau}(s)|^{p-1}u_{\tau}(s)$.

For fixed $\xi \in \mathbb{R}^d$, we have, as $\tau \to 0$, 
\begin{align*}
	\lambda_{\tau}^{+} 
	&= \frac{-1+ \sqrt{1-4\tau |\xi|^2}}{2\tau}=\frac{2|\xi|^2}{-1-\sqrt{1-4\tau|\xi|^2}}\to -|\xi|^2,
	\\
	\lambda_{\tau}^{-} 
	&= \frac{-1- \sqrt{1-4\tau |\xi|^2}}{2\tau} \to -\infty,
\end{align*}
and
\begin{align*}
	\frac{1}{\lambda_{\tau}^{-} - \lambda_{\tau}^{+}}
	= -\frac{\tau}{ \sqrt{1-4\tau |\xi|^2}}  \to 0
\end{align*}
\begin{align*}
	\frac{\lambda_{\tau}^{\pm}}{ \lambda_{\tau}^{-} -  \lambda_{\tau}^{+}}
%	= \frac{-1\pm \sqrt{1-4\tau |\xi|^2}}{2\tau} \times \l(-\frac{\tau}{ \sqrt{1-4\tau |\xi|^2}} \r) 
	=- \frac{-1\pm \sqrt{1-4\tau |\xi|^2}}{2\sqrt{1-4\tau |\xi|^2}}
	\to 
	\begin{cases}
	0 & \text{ if } \lambda_{\tau}^{+}
	\\
	1 & \text{ if } \lambda_{\tau}^{-}
	\end{cases}
\end{align*}
\begin{align*}
	\frac{\lambda_{\tau}^{-}\lambda_{\tau}^{+}}{ \lambda_{\tau}^{-} -  \lambda_{\tau}^{+}}
%	= - \frac{-1- \sqrt{1-4\tau |\xi|^2}}{2\sqrt{1-4\tau |\xi|^2}} \times \frac{-1+ \sqrt{1-4\tau |\xi|^2}}{2\tau}
	=-\frac{|\xi|^2}{\sqrt{1-4\tau |\xi|^2}} \to -|\xi|^2.
\end{align*}
Thus, roughly, we have
\begin{align*}
	&\frac{1}{\tau}  \cD_{\tau}(t) + \partial_t \cD_{\tau}(t)
	=  \frac{1}{ \lambda_{\tau}^{-} -  \lambda_{\tau}^{+}} 
	\l(\lambda_{\tau}^{-} e^{t \lambda_{\tau}^{+}}
	- \lambda_{\tau}^{+} e^{t \lambda_{\tau}^{-}} \r)
	\to e^{t\Delta},
	\\
	&\cD_{\tau}(t)= \frac{1}{ \lambda_{\tau}^{-} -  \lambda_{\tau}^{+}} ( -e^{t \lambda_{\tau}^{+}}+e^{t \lambda_{\tau}^{-}})
	\to 0,
\end{align*}
and
\begin{align*}
	\frac{1}{\tau}  \cD_{\tau}(t) = \frac{1}{ \tau(\lambda_{\tau}^{-} -  \lambda_{\tau}^{+})} ( -e^{t \lambda_{\tau}^{+}}+e^{t \lambda_{\tau}^{-}})
	\to e^{t\Delta}. 
\end{align*}

From the observation of the symbols, we find that the low frequency part of the propagator is like the heat propagator and the high part is like the wave propagator with exponential decay. Indeed, we have
\begin{align*}
	\cD_{\tau}(t) 
	=\tau e^{-\frac{t}{2\tau}} \cF^{-1} L_{\tau}(t,\xi) \cF,
\end{align*}
and
\begin{align*}
	L_{\tau}(t,\xi) = 
	\begin{cases}
	\frac{2}{\sqrt{1-4\tau|\xi|^2}} \sinh \l(t \frac{\sqrt{1-4\tau |\xi|^2}}{2\tau}\r) & \text{if }  |\xi| \leq \frac{1}{2\sqrt{\tau}}, 
	\\
	\frac{2}{\sqrt{4\tau|\xi|^2-1}} \sin \l(t \frac{\sqrt{4\tau |\xi|^2-1}}{2\tau}\r) & \text{if } |\xi| > \frac{1}{2\sqrt{\tau}}.
	\end{cases}
\end{align*}
Therefore, we need to calculate low and high parts respectively. 

%%%%%%%%%%%%%%%%%%%%%%%%%%%%%%%%%%%%%%%%%%%%%%%%%%%%%%%%
\subsection{Lemmas}

It is easily seen that $\phi_{\tau}(t,x) = \phi_1(\tau^{-1}t,\tau^{-1/2}x)$, where $\phi_1$ is a solution of $\partial_t^2 \phi_1 + \partial_t \phi_1 -\Delta \phi_1 =0$. Therefore, by this scaling and the previous result \cite{Inu19, InWa19}, we have the following Strichartz eatimates. 

\begin{lemma}[Homogeneous estimate for low frequency]
Let $s\in \mathbb{R}$, $q \in [2,\infty]$, $r\in [2,\infty]$. % and $1/q \leq d/2 \times (1/2-1/r)$. 
Assume that $(q,r)$ satisfies
\begin{align*}
	\frac{d}{2} \left(\frac{1}{2} - \frac{1}{r} \right) \geq \frac{1}{q}.
\end{align*}
Then, we have the following.
\begin{align*}
	\left\| \frac{1}{\tau}\cD_{\tau}(t) \chi_{\leq \tau^{-1/2}} f \right\|_{L_t^q \dot{B}^{s}_{r,2}(I)} 
	\cleq \tau^{ \alpha(q,r) } \| \chi_{\leq \tau^{-1/2}} f\|_{\dot{B}^{s}_{2,2}},
\end{align*}
and 
\begin{align*}
	\norm{\partial_t \cD_{\tau}(t) \chi_{\leq \tau^{-1/2}}  f}_{L_t^q \dot{B}^{s}_{r,2}(I)} 
	\cleq \tau^{\alpha(q,r)} \norm{ \chi_{\leq \tau^{-1/2}} f}_{\dot{B}^{s}_{2,2}},
\end{align*}
where the implicit constants are independent of $\tau$. Moreover, we also have
\begin{align*}
	\left\|\left(\frac{1}{\tau} \partial_t \mathcal{D}_{\tau}(t) + \partial_t^2 \mathcal{D}_{\tau}(t)\right) \chi_{\leq \tau^{-1/2}} f \right\|_{L_t^{\infty}L_x^2} 
	\lesssim \| |\nabla|^2 \chi_{\leq \tau^{-1/2}}f\|_{L^2}
	\lesssim \tau^{-\frac{1}{2}}\|f\|_{H^1}.
\end{align*}
\end{lemma}

\begin{lemma}[Homogeneous estimate for high frequency]
\label{lem2.2.0}
Let $s\in \mathbb{R}$, $q \in [2,\infty]$, and $r\in [2,\infty]$. % and $\gamma= \max\{ d(1/2-1/r) -1/q, (d+1)/2 \times (1/2-1/r) \}$. 
Then, we have the following.
\begin{align*}
	\norm{ \frac{1}{\tau} \cD_{\tau}(t) \chi_{>\tau^{-1/2} }  f}_{L_t^q \dot{B}^{s}_{r,2}(I)} 
	&\cleq \tau^{ \alpha(q,r) + \frac{\gamma(q,r)}{2} -\frac{1}{2} } \norm{ |\nabla|^{\gamma(q,r)-1} \chi_{> \tau^{-1/2}} f}_{\dot{B}^{s}_{2,2}},
%	\\
%	&\cleq \tau^{ -\frac{d}{2}\l(\frac{1}{2} - \frac{1}{r}\r)+\frac{1}{q} + \frac{\gamma}{2}} \norm{ |\nabla|^{s+\gamma}\chi_{> 1/\sqrt{\tau}}(\nabla) f}_{L^2},
\end{align*}
and
\begin{align*}
	\norm{\partial_t \cD_{\tau}(t) \chi_{> \tau^{-1/2}}  f}_{L_t^q \dot{B}^{s}_{r,2}(I)} 
	\cleq \tau^{ \alpha(q,r)  + \frac{\gamma(q,r)}{2}} \norm{ |\nabla|^{\gamma(q,r)}  \chi_{>\tau^{-1/2}} f}_{\dot{B}^{s}_{2,2}},
\end{align*}
where the implicit constants are independent of $\tau$. Moreover, we also have
\begin{align*}
	\left\|\left(\frac{1}{\tau} \partial_t \mathcal{D}_{\tau}(t) + \partial_t^2 \mathcal{D}_{\tau}(t)\right) \chi_{>\tau^{-1/2}} f \right\|_{L_t^{\infty}L_x^2} 
	\lesssim \tau^{-\frac{1}{2}}\|f\|_{H^1}.
\end{align*}

\end{lemma}

Moreover, we also have the inhomogeneous Strichartz estimate.

\begin{lemma}[Inhomogeneous estimate for low frequency]
Let $s\in \mathbb{R}$, $q \in [2,\infty]$, and $r\in [2,\infty]$. Assume that  $(q,r)$ and $(\tilde{q},\tilde{r})$ satisfy 
\begin{align}
\label{eq2.1.0}
	\frac{d}{2}\l( \frac{1}{2} - \frac{1}{r}\r) \geq \frac{1}{q} \text{ and } \frac{d}{2}\l( \frac{1}{2} - \frac{1}{\tilde{r}}\r)
	\geq  \frac{1}{\tilde{q}},
\end{align}
and additionally assume $\tilde{q}' < q$ if the both equalities hold.
Then we have
\begin{align*}
	\norm{ \int_0^t \frac{1}{\tau} \cD_{\tau}(t-s) \chi_{\leq \tau^{-1/2}} F(s) ds }_{L_t^q \dot{B}^{s}_{r,2}(I)} 
%	\\
%	&\qquad  
	\cleq
	\tau^{\alpha(q,r) + \alpha(\tilde{q},\tilde{r}) } 
	 \norm{ \chi_{\leq \tau^{-1/2}} F}_{L_t^{\tilde{q}'} \dot{B}^{s}_{\tilde{r}',2}},
\end{align*}
where the implicit constant is independent of $\tau$. Moreover, we also have
\begin{align*}
	\left\| \int_0^t \frac{1}{\tau} \partial_t \cD_{\tau}(t-s) \chi_{\leq \tau^{-1/2}} F(s) ds \right\|_{L_t^{\infty}L_x^2} 
	\lesssim \tau^{-1+\alpha(\tilde{q},\tilde{r}) +\frac{s}{2}}\| \chi_{\leq \tau^{-1/2}}F\|_{L_t^{\tilde{q}'}W_x^{s,\tilde{r}'}}
\end{align*}
for $s \in [0,1]$.
\end{lemma}

\begin{remark}
In \cite{Inu19,InWa19}, we assume that $1<\tilde{q}' < q<\infty$ if the both equalities \eqref{eq2.1.0} hold. In the above lemma, we may take $q=\infty$ or $\tilde{q}=\infty$. See Appendix \ref{appA} for the proof.  
\end{remark}

\begin{lemma}[Inhomogeneous estimate for high frequency]
\label{lem2.4.0}
Let $s\in \mathbb{R}$, $q, \tilde{q} \in [2,\infty]$, and $r, \tilde{r}\in [2,\infty]$. 
Then, we have the following.
\begin{align*}
	&\norm{ \int_0^t \frac{1}{\tau}\cD_{\tau}(t-s) \chi_{> \tau^{-1/2} } F(s) ds}_{L_t^q \dot{B}^{s}_{r,2}(I)} 
	\\
	&\qquad \cleq
	\tau^{ \alpha(q,r) + \frac{\gamma(q,r)}{2} + \alpha(\tilde{q},\tilde{r}) + \frac{\gamma(\tilde{q},\tilde{r})}{2} +\frac{\delta}{2}-\frac{1}{2}}
	 \norm{ |\nabla|^{\gamma(q,r)+\gamma(\tilde{q},\tilde{r})+\delta-1} \chi_{> \tau^{-1/2} }F}_{L_t^{\tilde{q}'} \dot{B}^{s}_{\tilde{r}',2}(I)}
%	 \\
%	&\qquad \cleq
%	\tau^{ -\frac{d}{2}\l(\frac{1}{2} - \frac{1}{r}\r)+\frac{1}{q} + \frac{\gamma}{2} -\frac{d}{2}\l(\frac{1}{2} - \frac{1}{\tilde{r}}\r)+\frac{1}{\tilde{q}} + \frac{\tilde{\gamma}}{2}+\frac{\delta}{2}}
%	 \norm{ |\nabla|^{\sigma+\gamma+\tilde{\gamma}+\delta} \chi_{> 1/\sqrt{\tau}}( \nabla) F}_{L_t^{\tilde{q}'} L_x^{\tilde{r}'}(I)}
\end{align*}
where the implicit constant is independent of $\tau$ and we set $\delta:=\delta((q,r),(\tilde{q},\tilde{r}))$, which is defined in Table \ref{tab1} below. 
Moreover, we also have
\begin{align*}
	\left\| \int_0^t \frac{1}{\tau} \partial_t \cD_{\tau}(t-s) \chi_{> \tau^{-1/2}} F(s) ds \right\|_{L_t^{\infty}L_x^2} 
	\lesssim \tau^{-1+\alpha(\tilde{q},\tilde{r})+\frac{\gamma(\tilde{q},\tilde{r})}{2}}\||\nabla|^{\gamma(\tilde{q},\tilde{r})} \chi_{> \tau^{-1/2} }F\|_{L_t^{\tilde{q}'}W_x^{s,\tilde{r}'}}.
\end{align*}

\begin{table}[htb]
\begingroup
\renewcommand{\arraystretch}{1.6}
\begin{tabular}{|c|c|c|} 
	\hline
	$\delta((q,r),(\tilde{q},\tilde{r}))$
		&  $\frac{1}{\tilde{q}} \l( \frac{1}{2}- \frac{1}{r}\r) < \frac{1}{q} \l( \frac{1}{2} - \frac{1}{\tilde{r}} \r) $
		&  $\frac{1}{\tilde{q}} \l( \frac{1}{2}- \frac{1}{r}\r) > \frac{1}{q} \l( \frac{1}{2} - \frac{1}{\tilde{r}} \r)$ 
	\\[4pt]
	\hline \hline
	$\frac{d-1}{2}\l( \frac{1}{2}- \frac{1}{r}\r) \geq \frac{1}{q}$ 
	& $0$
	& $0$
	\\
	$\frac{d-1}{2}\l( \frac{1}{2}- \frac{1}{\tilde{r}}\r) \geq \frac{1}{\tilde{q}}$
	&  
	& 
	\\[5pt]
	\hline
	$\frac{d-1}{2}\l( \frac{1}{2}- \frac{1}{r}\r) \geq \frac{1}{q}$ 
	& $\times$
	& $\frac{\tilde{q}}{q}  \l\{ \frac{1}{\tilde{q}}-\frac{d-1}{2} \l( \frac{1}{2} - \frac{1}{\tilde{r}}\r) \r\}$
	\\
	$\frac{d-1}{2}\l( \frac{1}{2}- \frac{1}{\tilde{r}}\r) < \frac{1}{\tilde{q}}$
	&  
	& 
	\\[5pt]
	\hline
	$\frac{d-1}{2}\l( \frac{1}{2}- \frac{1}{r}\r) < \frac{1}{q}$ 
	& $\frac{q}{\tilde{q}}  \l\{ \frac{1}{q}-\frac{d-1}{2} \l( \frac{1}{2} - \frac{1}{r}\r) \r\}$
	& $\times$
	\\
	$\frac{d-1}{2}\l( \frac{1}{2}- \frac{1}{\tilde{r}}\r) \geq \frac{1}{\tilde{q}}$
	&  
	& 
	\\[5pt]
	\hline
	$\frac{d-1}{2}\l( \frac{1}{2}- \frac{1}{r}\r) < \frac{1}{q}$ 
	& $\frac{1}{\tilde{q}} \frac{d-1}{2} \l\{ \tilde{q}\l( \frac{1}{2}- \frac{1}{\tilde{r}}\r) - q\l( \frac{1}{2} - \frac{1}{r}\r) \r\}$
	& $\frac{1}{q} \frac{d-1}{2} \l\{ q\l( \frac{1}{2}- \frac{1}{r}\r) - \tilde{q}\l( \frac{1}{2} - \frac{1}{\tilde{r}}\r) \r\}$
	\\
	$\frac{d-1}{2}\l( \frac{1}{2}- \frac{1}{\tilde{r}}\r)< \frac{1}{\tilde{q}}$
	&
	&
	\\
	\hline
\end{tabular}
\caption{The value of $\delta$. ($\times$ means that the case does not occur.)} \label{tab1}
\endgroup
\end{table} 
\end{lemma}

\begin{remark}
In \cite{Inu19}, the first author assumed that $d \geq 2$ in the Strichartz estimates for the high frequency part. However, they holds in the one dimensional case. See Appendix \ref{appA.2}. 
\end{remark}

We use the following lemma without notice. 

\begin{lemma}
For $\sigma,a \geq 0$ and $1<p<\infty$, we have
\begin{align*}
	\norm{ |\nabla|^{\sigma} \chi_{\leq \tau^{-1/2} }  f}_{L^p}
	\leq \tau^{-a/2} \norm{ |\nabla|^{\sigma-a} \chi_{\leq  \tau^{-1/2} }  f}_{L^p}
	\\
	\norm{ |\nabla|^{\sigma} \chi_{> \tau^{-1/2} }  f}_{L^p}
%	&=\norm{\chi_{> 1/\sqrt{\tau}}  f}_{ \tau^a X(|\nabla|^b)}
%	\\
%	&\leq \norm{\chi_{> 1/\sqrt{\tau}}  f}_{ \tau^{a-c/2} X(|\nabla|^{b+c})}
	\leq \tau^{a/2} \norm{ |\nabla|^{\sigma+a} \chi_{>  \tau^{-1/2} }  f}_{L^p}
\end{align*}
\end{lemma}

\begin{proof}
These inequalities follow immediately from $\chi_{\leq \tau^{-1/2} }$ and $\chi_{>  \tau^{-1/2} }$, respectively.
%Indeed, we have
%\begin{align*}
%	\norm{\chi_{> 1/\sqrt{\tau}}  f}_{ \tau^a X(|\nabla|^b)}
%	&=\tau^{-a} \norm{ |\nabla|^{b}\chi_{> 1/\sqrt{\tau}}  f}_{X}
%	\\
%	&=\tau^{-a+c/2-c/2} \norm{ |\nabla|^{b+c-c}\chi_{> 1/\sqrt{\tau}}  f}_{X}
%	\\
%	&\leq \tau^{-a+c/2} \norm{ |\nabla|^{b+c}\chi_{> 1/\sqrt{\tau}}  f}_{X}
%	\\
%	&= \norm{\chi_{> 1/\sqrt{\tau}}  f}_{ \tau^{a-c/2} X( |\nabla|^{b+c})}
%\end{align*}
%since $\tau^{-1/2}|\nabla|^{-1}\leq 1$.
\end{proof}

%Let $N \in \mathbb{N}$, $1\leq q_i, r_i \leq \infty$ for $i \in \{1,2,...,N\}$. We denote $X_i:= L_t^{q_i}L_x^{r_i}(I)$ for an interval $I$. Then, by Lemma 3.3 of [the second author--Nakanishi--Ozawa], we have the following corollary. 
We use the following lemmas by \cite{MNO02}. 

\begin{lemma}[{\cite[Lemma 3.3]{MNO02}}]
\label{lem3.6}
Let $I$ be a time interval. Take $X_i \in \mathcal{L}$ for $i \in \{0,1,2,...,N\}$ and assume that $|u(t,x)| \lesssim  \sum_{i=1}^{N} |v_i(t,x)|$ and $v_i \in X_i$. Then, we have
\begin{align*}
	\| u \|_{\sum_{i=0}^{N} X_i} \cleq \sum_{i=0}^{N} \| v_i \|_{X_i}.
\end{align*}
\end{lemma}

This lemma and the H\"{o}lder inequality imply the difference estimate of the nonlinear term. 

\begin{lemma}[{\cite[Lemma 3.4]{MNO02}}]
\label{lem3.7}
Let $I$ be a time interval and $\mathcal{N}(u)=\lambda |u|^{p-1}u$ with $p>1$. Take $X_i \in \mathcal{L}$ and $Z_i \in \mathcal{B}$ for $i=0,1,2,3$. Assume that $\sigma(Z_i) < \min \{2,p\}$ and $X_i^{p-1}Z_i \in \mathcal{B}$ for $i=0,...,3$.
Then we have
\begin{align}
\label{eq3.1.0}
	&\| \mathcal{N}(u) \|_{\sum_{i=0}^{3} X_i^{p-1}Z_{i}} 
	\\ \notag
	&\quad \cleq \inf_{u=a+b} (\| a \|_{X_0 \cap X_1} + \| b \|_{X_2 \cap X_3})^{p-1} (\| a \|_{Z_0 \cap Z_2} + \| b \|_{Z_1 \cap Z_3} ).
\end{align}
\end{lemma}

Combining Lemmas \ref{lem3.6} and \ref{lem3.7}, we also have the similar estimate to \eqref{eq3.1.0} for inhomogeneous Besov spaces. 

We also use the following refined version of Lemma \ref{lem3.7} by Masmoudi and Nakanishi \cite{MaNa02}.

\begin{lemma}[{\cite[Lemma 3.2]{MaNa02}}]
\label{lem3.8}
Let $I$ be a time interval and $\mathcal{N}(u)=\lambda |u|^{p-1}u$ with $p>1$. Take $X_i \in \mathcal{L}$ and $Z_i \in \mathcal{B}$ for $i=0,1,2,3$. Assume that $\sigma(Z_i) < \min \{2,p\}$ and $X_i^{p-1}Z_i \in \mathcal{B}$ for $i=0,...,3$. Then we have
\begin{align*}
	\norm{P_j \mathcal{N}(u) }_{\sum_{i=0}^{3} X_i^{p-1}Z_i}
	&\lesssim \inf_{u=a+b} \left( \|a\|_{X_0 \cap X_1} + \|b\|_{X_2 \cap X_3}  \right)^{p-1} 
	\\
	&\quad \times \kappa_j *_j (\| P_j  a \|_{Z_0\cap Z_2} + \| P_j  b \|_{Z_1 \cap Z_3}),
\end{align*}
where $*_j$ denotes the convolution related to $j$ and 
\begin{align*}
	\kappa_j := \max_{i=0,1,2,3} \min\{ 2^{ \frac{\sigma(Z_i) j}{2} }, 2^{ \frac{(\sigma(Z_i) -2) j}{2} }\}.
\end{align*}
\end{lemma}

%%%%%%%%%%%%%%%%%%%%%%%%%%%%%%%%%%%%%%%%%%%%%%%%%%%%%%%%
\subsection{Local existence and Uniformly bound}

In what follows, we only treat the case of $d\geq 3$. The cases $d=1,2$ are easier (see Section \ref{sec2.5}). 
%We set $p_1=1+4/(d-2)$.  

\subsubsection{Function spaces}
We define the function spaces as follows.
\begin{align*}
	\mathscr{E}&:=L_t^{\infty}H^1,
	\\
	\mathscr{H}_1&:=L_{t,x}^{\frac{2(d+2)}{d-2}},
	&
	\mathscr{H}_2&:= L_{t}^{\frac{2(d+2)}{d}} B_{\frac{2(d+2)}{d}, 2}^{1},
%	\\
%	\mathscr{H}_3&:=L_{t,x}^{\frac{2(d+2)}{d}},
	\\
	\mathscr{W}_1&:=\tau^{\frac{d-2}{4(d+1)}} L_{t,x}^{\frac{2(d+1)}{d-2}},
	&
	\mathscr{W}_2&:=\tau^{\frac{d-1}{4(d+1)}} L_{t}^{\frac{2(d+1)}{d-1}} B_{\frac{2(d+1)}{d-1}, 2}^{\frac{1}{2}},
	\\
	\mathscr{D}_1&:=\tau^{\frac{1}{4}} L_{t}^{\frac{2(d+2)}{d}} B_{\frac{2(d+2)}{d}, 2}^{\frac{1}{2}}.
\end{align*}
We also define the function spaces $\mathscr{R}_0, \cdots ,\mathscr{R}_3$ by 
\begin{align*}
	\mathscr{R}_0&:=L_t^{\frac{2(d+2)}{d+4}}W_x^{1, \frac{2(d+2)}{d+4}} 
	=\mathscr{H}_1^{\frac{4}{d-2}} \mathscr{H}_2,
	\\
	\mathscr{R}_1&:= \tau^{\frac{1}{4}} L_t^{\frac{2(d+2)}{d+4}}W_x^{\frac{1}{2}, \frac{2(d+2)}{d+4}}
	=\mathscr{H}_1^{\frac{4}{d-2}} \mathscr{D}_1,
	\\
	\mathscr{R}_2&:=\tau^{\frac{d+3}{4(d+1)}} L_{t}^{\frac{2(d+1)}{d+3}}W_{x}^{\frac{1}{2}, \frac{2(d+1)}{d+3}}
	=\mathscr{W}_1^{\frac{4}{d-2}} \mathscr{W}_2,
	\\
	\mathscr{R}_3&:=\tau^{\frac{1}{d+1}}L_{t}^{\frac{2(d+1)(d+2)}{d^2+5d+8}}W_{x}^{1, \frac{2(d+1)(d+2)}{d^2+5d+8}}
	=\mathscr{W}_1^{\frac{4}{d-2}} \mathscr{H}_2.
\end{align*}
We will use the following function spaces to construct a contraction mapping. 
\begin{align*}
	\mathscr{Z}_{\theta}&:= [\mathscr{E},\mathscr{H}_2]_{\theta} | [\mathscr{E},\mathscr{W}_2]_{\theta} \cap [\mathscr{E},\mathscr{D}_1]_{\theta}
	\\
	\mathscr{X}&:=\mathscr{H}_1 | \mathscr{W}_1.
\end{align*}
%where $s$ satisfies
%%\begin{align*}
%%	&\max\l\{ \frac{d+1}{2(d+2)},  \frac{d-1}{2d(d+1)}+\frac{d+1}{2(d+2)},  \frac{d-2}{2d(d+1)}+\frac{d+1}{2(d+2)}\r\}
%%	\\
%%	&\leq s 
%%	\leq 
%%	\min\l\{1- \frac{d+1}{2(d+2)} -\frac{d-1}{2d(d+1)},1-\frac{d+1}{2(d+2)}- \frac{d^2+5d+4}{2(d+1)(d+2)}- \frac{d^2+d-4}{2d(d+1)(d+2)} \r\}
%%\end{align*}
%\begin{align*}
%	\frac{d-1}{2d(d+1)}+\frac{d+1}{2(d+2)}
%	\leq s 
%	\leq 
%	1- \frac{d+1}{2(d+2)} -\frac{d-1}{2d(d+1)}
%\end{align*}
For the reader's convenience, we give an explanation of the exponents of the function spaces. 
We set the exponents $(q,r)$ by 
\begin{align*}
	\mathfrak{e}&:=(\infty,2), 
	\\
	\mathfrak{h}_1&:=\l(\frac{2(d+2)}{d-2},\frac{2(d+2)}{d-2}\r), 
	&\mathfrak{h}_2:=\l(\frac{2(d+2)}{d},\frac{2(d+2)}{d}\r),
	\\
	\mathfrak{w}_1&:=\l(\frac{2(d+1)}{d-2},\frac{2(d+1)}{d-2}\r),
	&\mathfrak{w}_2:=\l(\frac{2(d+1)}{d-1},\frac{2(d+1)}{d-1}\r),
	\\
	\mathfrak{w}_3 & :=\left( \frac{2(d+1)(d+2)}{d^2+d-4}, \frac{2(d+1)(d+2)}{d^2+d-4}\right).
\end{align*}
The exponent $\mathfrak{e}$ is related to the energy space $\mathscr{E}$, $\mathfrak{h}_1$ is $\mathscr{H}_1$, $\mathfrak{h}_2$ is $\mathscr{H}_2$ and $\mathscr{D}_1$, $\mathfrak{w}_1$ is $\mathscr{W}_1$, $\mathfrak{w}_2$ is $\mathscr{W}_2$. 
The exponents of $\mathscr{R}_0$ and $\mathscr{R}_1$ is  $\mathfrak{h}_2'$, $\mathscr{R}_2$ is related to $\mathfrak{w}_2'$, $\mathscr{R}_3$ is $\mathfrak{w}_3'$, where $\mathfrak{p}'=(p',p')$.

$\mathfrak{h}_2$ satisfies $\frac{d}{2}\left( \frac{1}{2} - \frac{1}{r}\right)=\frac{1}{q}$, 
$\mathfrak{w}_2$ satisfies $\frac{d-1}{2}\left( \frac{1}{2} - \frac{1}{r}\right)=\frac{1}{q}$, 
$\mathfrak{w}_1$ and $\mathfrak{w}_3$ satisfies $\frac{d-1}{2}\left( \frac{1}{2} - \frac{1}{r}\right) > \frac{1}{q}$.
%and
%$\mathfrak{w}_3$ satisfies $\frac{d-1}{2}\left( \frac{1}{2} - \frac{1}{r}\right) > \frac{1}{q}$. 
Namely, $\mathfrak{h}_2$ is related to the heat admissible pair and $\mathfrak{w}_1$ and $\mathfrak{w}_2$ 
are related to the wave admissible pair. 
$\mathfrak{h}_1$ satisfies $\frac{d-1}{2}\left( \frac{1}{2} - \frac{1}{r}\right) > \frac{1}{q}$. 
However, since the Sobolev embedding $\| f \|_{L_{x}^{\frac{2(d+2)}{d-2}}} \cleq \| f \|_{\dot{W}_{x}^{1,\frac{2d(d+2)}{d^2+4}}}$ holds,  its new exponent $\widetilde{\mathfrak{h}}_1 := \l(\frac{d-2}{2(d+2)},\frac{d^2+4}{2d(d+2)}\r)$ lies on the heat line $\frac{d}{2}\left( \frac{1}{2} - \frac{1}{r}\right)=\frac{1}{q}$. Thus, $\mathfrak{h}_1$ is related to the heat admissible pair.

%We set the exponent related to $\mathscr{R}_3$ by 
%%\begin{align*}
%%	\mathfrak{w}_3 = (w_3,w_3) :=\left( \frac{2(d+1)(d+2)}{d^2+d-4}, \frac{2(d+1)(d+2)}{d^2+d-4}\right).
%%\end{align*}

We collect the value of $\alpha$, $\gamma$, and $\delta$ for these exponents in the following tables.

\begin{table}[htb]
\begingroup
\renewcommand{\arraystretch}{1.3}
\begin{tabular}{|c||c|c|c|c|c|c|} 
\hline
  &  $\mathfrak{e}$ & $\widetilde{\mathfrak{h}}_1$ & $\mathfrak{h}_2$ & $\mathfrak{w}_1$ & $\mathfrak{w}_2$ & $\mathfrak{w}_3$
\\ \hline \hline
$\alpha$	 & 0 & 0 & 0 & $-\frac{d+4}{4(d+1)}$ & $-\frac{1}{2(d+1)}$ & $-\frac{1}{d+1}$ \\ \hline
\end{tabular}
\caption{The value of $\alpha$ for the exponents. }
\label{tab2}
\endgroup
\end{table}

\begin{table}[htb]
\begingroup
\renewcommand{\arraystretch}{1.3}
\begin{tabular}{|c||c|c|c|c|c|} 
\hline
  &  $\mathfrak{e}$  & $\mathfrak{h}_2$ & $\mathfrak{w}_1$ & $\mathfrak{w}_2$ & $\mathfrak{w}_3$
\\ \hline \hline
$\gamma$	 & 0 &  $\frac{d+1}{2(d+2)}$ & $1$ & $\frac{1}{2}$ & $\frac{d+4}{2(d+2)}$ \\ \hline
\end{tabular}
\caption{The value of $\gamma$ for the exponents. }
\label{tab3}
\endgroup
\end{table}
%
%We do not need to calculate $\gamma(\widetilde{\mathfrak{h}}_1)$ since $\gamma$ appears only in the high frequency estimate and we do not apply $\mathscr{H}_1$ to the high frequency.

\begin{table}[htb]
\begingroup
\renewcommand{\arraystretch}{1.3}
\begin{tabular}{|c||c|c|c|c|} 
	\hline
	$X \setminus Y$
		&  $\mathfrak{e}$
		&  $\mathfrak{h}_2$ 
		&  $\mathfrak{w}_2$ 
		&  $\mathfrak{w}_3$ 
	\\[4pt]
	\hline \hline 
	$\mathfrak{e}$
		& $0$
		& $0$
		& $0$
		& $0$
	\\[4pt] \hline
	$\mathfrak{h}_2$
		& $0$
		& $0$
		& $\frac{d-1}{2d(d+1)}$
		& $\frac{d^2+d-4}{2d(d+1)(d+2)}$
	\\[4pt] \hline
	$\mathfrak{w}_1$
		& $0$
		& $\frac{d-2}{2d(d+1)}$
		& $0$
		& $0$
	\\[4pt] \hline
	$\mathfrak{w}_2$
		& $0$
		& $\frac{d-1}{2d(d+1)}$
		& $0$
		& $0$
	\\[4pt] \hline
\end{tabular}
\caption{The value of $\delta(X,Y)$.}
\endgroup
\end{table}

%%%%%%%%%%%%%%%%%%%%%%%%%%%%%%%%%%%%%%%%%%%%%%%%%%%%%%%%

\subsubsection{Uniform boundedness}

%We will show that $\Phi$ is a contraction mapping on $\mathscr{E} \cap \mathscr{Z}_{\theta} \cap \mathscr{X}$. 

We show the uniform boundedness. Since we have
\begin{align*}
	u_{\tau} = \l( \frac{1}{\tau} \mathcal{D}_{\tau}(t) + \partial_t \mathcal{D}_{\tau}(t) \r) f_{\tau}
	+ \mathcal{D}_{\tau}(t)  g_{\tau}
	+ \int_{0}^{t} \frac{1}{\tau} \mathcal{D}_{\tau} (t-s) \mathcal{N}(u_{\tau})(s)ds,
\end{align*}
where we recall that $\mathcal{N}(u)=\mu |u|^{p-1}u$, 
it holds from the homogeneous Strichartz estimates that
\begin{align*}
	\|  u_{\tau} \|_{\mathscr{E} \cap \mathscr{Z}_{\theta} \cap \mathscr{X}}
	\cleq \| f_{\tau} \|_{H^1} + \tau^{\frac{1}{2}} \| g_{\tau} \|_{L^2} 
	+\left\| \int_{0}^{t} \frac{1}{\tau} \mathcal{D}_{\tau} (t-s) \mathcal{N}(u_{\tau})(s)ds  \right\|_{\mathscr{E} \cap \mathscr{Z}_{\theta} \cap \mathscr{X}}.
\end{align*}
By the inhomogeneous Strichartz estimates, we obtain the following estimate.
\begin{align*}
	\left\| \int_{0}^{t} \frac{1}{\tau} \mathcal{D}_{\tau} (t-s) \mathcal{N}(u_{\tau})(s)ds  \right\|_{\mathscr{E} \cap \mathscr{Z}_{\theta} \cap \mathscr{X}}
	\cleq T^{1-\theta} \| \mathcal{N}(u_{\tau})\|_{\sum_{i=0}^{3}[\mathscr{E},\mathscr{R}_i]_{\theta}}.
\end{align*}
%for $\theta \in [0,1]$. 
Then,  by Lemmas \ref{lem3.6} and \ref{lem3.7}, we have
\begin{align*}
	\norm{\mathcal{N}(u)}_{\sum_{i=0}^{3}[\mathscr{E},\mathscr{R}_i]_{\theta}}
%	&\cleq \inf_{u=u_1+u_2} (\| u_1 \|_{\mathscr{H}_1} + \| u_2 \|_{\mathscr{W}_1})^{p-1} (\| u_1 \|_{\mathscr{H}_2} + \| u_2 \|_{\mathscr{D}_1 \cap \mathscr{W}_1})
%	\\
%	&\cleq (\| \chi_{\leq 1/\sqrt{\tau}} u  \|_{\mathscr{H}_1} + \|  \chi_{> 1/\sqrt{\tau}} u  \|_{\mathscr{W}_1})^{p-1} (\|  \chi_{\leq 1/\sqrt{\tau}} u  \|_{\mathscr{H}_2} + \|  \chi_{> 1/\sqrt{\tau}} u  \|_{\mathscr{D}_1 \cap \mathscr{W}_1})
%	\\
%	&\cleq (\| u  \|_{\mathscr{H}_1}^{p-1} + \| u  \|_{\mathscr{W}_1}^{p-1}) \| u  \|_{\mathscr{Z}_{\theta}}
%	\\
	&\cleq \| u  \|_{\mathscr{X}}^{p-1} \| u  \|_{\mathscr{Z}_{\theta}}.
\end{align*}
Therefore, we obtain %since $\limsup_{\tau \to 0} (\|f_{\tau}\|_{H^1} + \tau^{1/2}\|g_{\tau}\|_{L^2})<\infty$, there exist $C_{0}>0$ and $\tau_{0}>0$ such that
\begin{align}
\label{eq3.0.0}
	\|  u_{\tau} \|_{\mathscr{E} \cap \mathscr{Z}_{\theta} \cap \mathscr{X}[0,T]}
	&\lesssim \| f_{\tau} \|_{H^1} + \tau^{\frac{1}{2}} \| g_{\tau} \|_{L^2}
	+ T^{1-\theta}
	 \|  u_{\tau} \|_{\mathscr{E} \cap \mathscr{Z}_{\theta} \cap \mathscr{X}[0,T]}^{p}
	% \| u_{\tau}  \|_{\mathscr{X}}^{p-1} \| u_{\tau}  \|_{\mathscr{Z}_{\theta}}.
%	\\ \notag
%	&\lesssim C_{0}
%	+ T^{1-\theta} \|  u_{\tau} \|_{\mathscr{E} \cap \mathscr{Z}_{\theta} \cap \mathscr{X}[0,T]}^{p}
\end{align}
By the Strichartz estimate and Lemmas \ref{lem3.6} and \ref{lem3.7}, we also have
\begin{align}
\label{eq3.0.1}
	\tau^{\frac{1}{2}}\| \partial_t u_{\tau} \|_{L_t^{\infty} L_x^2}
%	&\lesssim \|f_{\tau}\|_{H^1} + \|g_{\tau}\|_{L^2}
%	\\
%	&\quad + \tau^{\frac{1}{2}-\delta -\frac{d}{2}\left(\frac{1}{r_0} - \frac{1}{2}\right)} T^{\delta} \| u_{\tau} \|_{\mathscr{E}[0,T]}^{p} + \tau^{\frac{1}{2}}T^{1-\theta} \| u  \|_{\mathscr{X}}^{p-1} \| u  \|_{\mathscr{Z}_{\theta}}
	\lesssim  \|f_{\tau}\|_{H^1} + \tau^{\frac{1}{2}} \|g_{\tau}\|_{L^2} +T^{1-\theta} \|  u_{\tau} \|_{\mathscr{E} \cap \mathscr{Z}_{\theta} \cap \mathscr{X}[0,T]}^{p}.
\end{align}
%since $\kappa:=\frac{1}{2}-\delta -\frac{d}{2}\left(\frac{1}{r_0} - \frac{1}{2}\right)>0$ for small $\delta>0$.
By \eqref{eq3.0.0} and \eqref{eq3.0.1}, since $\limsup_{\tau \to 0} (\|f_{\tau}\|_{H^1} + \tau^{1/2}\|g_{\tau}\|_{L^2})<\infty$, there exist $C_{0}>0$ and $\tau_{0}>0$ such that 
\begin{align}
\label{eq3.0}
	\|  u_{\tau} \|_{\mathscr{E} \cap \mathscr{Z}_{\theta} \cap \mathscr{X}[0,T]} + \tau^{\frac{1}{2}}\| \partial_t u_{\tau} \|_{L_t^{\infty} L_x^2[0,T]}
	&\lesssim C_{0}
	+ T^{1-\theta} \|  u_{\tau} \|_{\mathscr{E} \cap \mathscr{Z}_{\theta} \cap \mathscr{X}[0,T]}^{p}
\end{align} 
for any $T < T_{\tau}^{*}$ and $0<\tau<\tau_{0}$. 

Suppose that $\liminf_{\tau\to 0}T_{\tau}^{*}=0$. Then, there exists a sequence $\{\tau_n\} \subset (0,\tau_{0})$ such that $T_{\tau_n}^{*} \to 0$. For sufficiently large $n$, it holds from \eqref{eq3.0} that
\begin{align}
\label{eq3.0.1}
	\|  u_{\tau_n} \|_{\mathscr{E} \cap \mathscr{Z}_{\theta} \cap \mathscr{X}[0,T]} 
	+ \tau_n^{\frac{1}{2}}\| \partial_t u_{\tau_n} \|_{L_t^{\infty} L_x^2[0,T]}
	\lesssim C_{0}
	%\sup_{\tau>0} (\| f_{\tau} \|_{H^1} + {\tau}^{\frac{1}{2}} \| g_{\tau} \|_{L^2})
\end{align}
for any $T<T_{\tau_n}^{*}$. This estimate and the blow-up alternative imply a contradiction. Therefore, we have $T^{*}=\liminf_{\tau\to 0}T_{\tau}^{*}>0$ and thus, for any $T<T^{*}$, there exist $C_T>0$ and $\tau_T$ such that $\|  u_{\tau} \|_{\mathscr{E} \cap \mathscr{Z}_{\theta} \cap \mathscr{X}[0,T]} + \tau^{1/2}\| \partial_t u_{\tau} \|_{L_t^{\infty} L_x^2[0,T]}\lesssim C_{T}$ for $\tau <\tau_T$. Consequently, we obtain Theorem \ref{thm1.1}. 
%Set 
%\begin{align*}
%	T^{*}:= \sup \{T \leq \liminf_{\tau\to 0}T_{\tau}^{*} :\exists C_T>0 \text{ s.t. } \|  u_{\tau} \|_{\mathscr{E} \cap \mathscr{Z}_{\theta} \cap \mathscr{X}[0,T]} < C_T \text{ for any } \tau \}
%\end{align*}
%It is easy to see $T^{*}>0$ from \eqref{eq3.0.1}. By its definition, we have $T^{*}\leq \liminf_{\tau\to 0}T_{\tau}^{*}$ and
%\begin{align*}
%	\|  u_{\tau} \|_{\mathscr{E} \cap \mathscr{Z}_{\theta} \cap \mathscr{X}[0,T]} < C_T \text{ for any } \tau
%\end{align*}
%if $T<T^{*}$. This completes the proof of Theorem \ref{thm1.1}.
%Suppose that $T'<T^{*}$. Then, there exists $T_0 \in (T',T^{*})$ such that for any $N\in \mathbb{N}$ there exists $\tau_N$ such that $ \|  u_{\tau_N} \|_{\mathscr{E} \cap \mathscr{Z}_{\theta} \cap \mathscr{X}[0,T_0]}>N$. This implies that $T_0 \geq T^{*}$ and a contradiction. Thus,  we have $T'=T^{*}$. This shows from the definition of $T'$ that 
%\begin{align*}
%	\|  u_{\tau} \|_{\mathscr{E} \cap \mathscr{Z}_{\theta} \cap \mathscr{X}[0,T]} < C_T
%\end{align*}
%holds for any $T<T^{*}$ and $\tau$. 

%%%%%%%%%%%%%%%%%%%%%%%%%%%%%%%%%%%%%%%%%%%%%%%%%%%%%%%%

\subsection{Proof of $L^2$-convergence}

In this section, we show that the $L^2$-convergence by a direct calculation. 

We set the function spaces without derivative by 
\begin{align*}
	\mathscr{E}^0&:=L_t^{\infty}L_x^2
	\\
	\mathscr{H}_2^0&:= L_{t}^{\frac{2(d+2)}{d}} L_x^{\frac{2(d+2)}{d}}
%	\\
%	\mathscr{H}_3^0&:=L_{t,x}^{\frac{2(d+2)}{d}}
%	\\
%	\mathscr{W}_2^0&:=\tau^{-\frac{1}{2(d+1)}} L_{t}^{\frac{2(d+1)}{d-1}} L^{\frac{2(d+1)}{d-1},}
%	\\
%	\mathscr{D}_1^0&:= L_{t}^{\frac{2(d+2)}{d}} L^{\frac{2(d+2)}{d}},
\end{align*}
and
\begin{align*}
	\mathscr{R}_0^0&:=L_t^{\frac{2(d+2)}{d+4}}L_x^{\frac{2(d+2)}{d+4}} 
	= \mathscr{H}_1^{\frac{4}{d-2}} \mathscr{H}_2^0
%	\\
%	\mathscr{R}_1^0&:= L_t^{\frac{2(d+2)}{d+4}}L_x^{\frac{2(d+2)}{d+4}}
%	&\mapsto H_1^{\frac{4}{d-2}} D_1^0
%	\\
%	\mathscr{R}_2^0&:%=\tau^{-1/4} \tau^{\frac{d+3}{4(d+1)}} L_{t}^{\frac{2(d+1)}{d+3}} L_{x}^{\frac{2(d+1)}{d+3}} 
%	=\tau^{\frac{1}{2(d+1)}} L_{t}^{\frac{2(d+1)}{d+3}} L_{x}^{\frac{2(d+1)}{d+3}} 
%	&\mapsto W_1^{\frac{4}{d-2}} W_2^0
	\\
	\mathscr{R}_3^0&:=\tau^{\frac{1}{d+1}}L_{t}^{\frac{2(d+1)(d+2)}{d^2+5d+8}}L_{x}^{\frac{2(d+1)(d+2)}{d^2+5d+8}}
	= \mathscr{W}_1^{\frac{4}{d-2}} \mathscr{H}_2^0
\end{align*}
Moreover, we set
\begin{align*}
	\mathscr{Z}_{\theta}^0&:= [\mathscr{E}^0,\mathscr{H}_2^0]_{\theta} | [\mathscr{E}^0,\mathscr{W}_2^0]_{\theta} \cap [\mathscr{E}^0,\mathscr{D}_1^0]_{\theta}
%	\\
%	\mathscr{X}^0&:=\mathscr{H}_1 | \mathscr{W}_1^0,
	\\
	\mathscr{Y}&:=[\mathscr{E}^{0},\mathscr{H}_2^0]_{\theta}
\end{align*}

We decompose $u_{\tau} - v$ as follows. 
\begin{align*}
	u_{\tau}(t) - v(t)
	&=\left( \frac{1}{\tau} \mathcal{D}_{\tau}(t) + \partial_t \mathcal{D}_{\tau}(t) \right) (f_{\tau}-f)
	\\
	&\quad +\left( \frac{1}{\tau} \mathcal{D}_{\tau}(t) + \partial_t \mathcal{D}_{\tau}(t) - e^{t\Delta} \right) f
	\\
	&\quad+ \mathcal{D}_{\tau}(t)g_{\tau}
	\\
	&\quad +\int_{0}^{t} \frac{1}{\tau} \mathcal{D}_{\tau}(t-s)(\mathcal{N}(u) -\mathcal{N}(v)) ds
	\\
	&\quad +\int_{0}^{t} \left( \frac{1}{\tau} \mathcal{D}_{\tau}(t-s)  - e^{(t-s)\Delta} \right) \mathcal{N}(v) ds
	\\
	&=:L_1+L_2+L_3+N_1+N_2
\end{align*}
We estimate the $\mathscr{Y}$-norm of these terms. 
First, we discuss the $\mathscr{Y}$-estimate of the homogeneous parts $L_1$, $L_2$, and $L_3$. By the embedding $\mathscr{Y} \supset \mathscr{E}^0 \cap \mathscr{H}_2^0$, it is enough to estimate $\mathscr{E}^0 \cap \mathscr{H}_2^0$-norms. 

\noindent{\bf The $\mathscr{Y}$-estimate of $L_3$:} 

By the Strichartz estimate, we obtain
\begin{align*}
	\| \mathcal{D}_{\tau} (t) g_{\tau} \|_{\mathscr{E}^0 \cap \mathscr{H}_2^0}
	&\lesssim  \left\| \cD_{\tau}(t)  \chi_{\leq  \tau^{-1/2} }  g_{\tau}\right\|_{\mathscr{E}^0 \cap \mathscr{H}_2^0}
	+ \left\| \cD_{\tau}(t) \chi_{> \tau^{-1/2} }  g_{\tau}\right\|_{\mathscr{E}^0 \cap \mathscr{H}_2^0}
	\\
	&\lesssim \tau \| g_{\tau}\|_{L^2}
	+ \tau^{\frac{1}{2} + \frac{\gamma(\mathfrak{h}_2)}{2} } \left\| |\nabla|^{\gamma(\mathfrak{h}_2)-1} \chi_{> \tau^{-1/2} }g_{\tau}\right\|_{L^2}
	\\
	&\lesssim \tau \|g_{\tau}\|_{L^2}
\end{align*}

\noindent{\bf The $\mathscr{Y}$-estimate of $L_1$:} 

In the same way as above, by dividing into the low and high frequency parts and applying the Strichartz estimates, we have
\begin{align*}
	\left\| \left( \frac{1}{\tau} \mathcal{D}_{\tau}(t) + \partial_t \mathcal{D}_{\tau}(t) \right) (f_{\tau}-f)  \right\|_{\mathscr{E}^0 \cap \mathscr{H}_2^0}
	&\lesssim \| f_{\tau}-f\|_{L^2}.
\end{align*} 

\noindent{\bf The $\mathscr{Y}$-estimate of $L_2$:} 

First, we consider the high frequency part $|\xi|> (8\tau)^{-1/2}$. The constant is not essential but  technical. % to estimate the low frequency part. %By the technical reason, we consider the high frequency part as $|\xi|>1/(4\sqrt{\tau})$. 
By the triangle inequality, we have 
\begin{align*}
	&\left\| \left( \frac{1}{\tau} \cD_{\tau}(t) + \partial_t \cD_{\tau}(t) - e^{t\Delta}\right)  \chi_{> (8\tau)^{-1/2} }  f \right\|_{\mathscr{E}^0 \cap \mathscr{H}_2^0}
	\\
	&\lesssim \left\| \left( \frac{1}{\tau} \cD_{\tau}(t) + \partial_t \cD_{\tau}(t) \right) \chi_{> (8\tau)^{-1/2} }  f \right\|_{\mathscr{E}^0 \cap \mathscr{H}_2^0}
	+ \left\|  e^{t\Delta} \chi_{> (8\tau)^{-1/2} }  f \right\|_{\mathscr{E}^0 \cap \mathscr{H}_2^0}
\end{align*}
By the Strichartz estimate, we obtain
\begin{align*}
	\left\| \left( \frac{1}{\tau} \cD_{\tau}(t) + \partial_t \cD_{\tau}(t) \right) \chi_{> (8\tau)^{-1/2} } f \right\|_{\mathscr{E}^0 \cap \mathscr{H}_2^0}
	&\lesssim  \left\| \chi_{> (8\tau)^{-1/2}  }  f \right\|_{L^2} 
%	+ \tau^{\frac{\gamma(\mathfrak{h})}{2}-\frac{1}{2}} \left\| \chi_{> (8\tau)^{-1/2} }  |\nabla|^{\gamma(\mathfrak{h})-1} f \right\|_{L^2}
%	+ \tau^{\frac{\gamma(\mathfrak{h})}{2}} \left\| \chi_{> (8\tau)^{-1/2} }  |\nabla|^{\gamma(\mathfrak{h})} f \right\|_{L^2}
	\\
	&\lesssim  \tau^{\frac{1}{2}} \left\| \chi_{> (8\tau)^{-1/2} } |\nabla| f \right\|_{L^2}
\end{align*}
since $\gamma(\mathfrak{h}_2)<1/2$. 
And by the Strichartz estimate for the heat propagator we have
\begin{align*}
	 \left\|  e^{t\Delta} \chi_{> (8\tau)^{-1/2} }  f \right\|_{\mathscr{E}^0 \cap \mathscr{H}_2^0}
	  \lesssim \left\| \chi_{> (8\tau)^{-1/2} }   f \right\|_{L_x^2}
	 \lesssim \tau^{\frac{1}{2}} \left\| \chi_{> (8\tau)^{-1/2} }  |\nabla| f \right\|_{L_x^2}.
\end{align*} 
It follows from these estimates that 
\begin{align*}
	\left\| \left( \frac{1}{\tau} \cD_{\tau}(t) + \partial_t \cD_{\tau}(t) - e^{t\Delta}\right)  \chi_{> (8\tau)^{-1/2} }  f \right\|_{\mathscr{E}^0 \cap \mathscr{H}_2^0}
	\lesssim  \tau^{\frac{1}{2}} \left\| \chi_{> (8\tau)^{-1/2} } |\nabla| f \right\|_{L_x^2}.
\end{align*}
Since $f \in H^1$, the right hand side is $o(\tau^{1/2})$. 

%Next, we consider the middle frequency part $\tau^{-1/4}< |\xi| < \tau^{-1/2}$. 
%
%Then, by the Strichartz estimates for the low frequency part $|\xi| < \tau^{-1/2}$, we obtain
%\begin{align*}
%	&\left\| \left( \frac{1}{\tau} \cD_{\tau}(t) + \partial_t \cD_{\tau}(t) - e^{t\Delta}\right)  \chi_{\tau^{-1/4}<\cdot<1/\tau^{1/2}}(\nabla) f \right\|_{\mathscr{E}^0 \cap \mathscr{H}_2^0}
%	\\
%	&\lesssim \| \chi_{\tau^{-1/4}<\cdot<1/\tau^{1/2}}(\nabla)  f\|_{L^2}
%	\\
%	&\lesssim \tau^{\frac{1}{4}} \| \chi_{\tau^{-1/4}<\cdot<1/\tau^{1/2}}(\nabla) f\|_{H^1}.
%\end{align*}

Next, we consider the low frequency part $|\xi| \leq (8\tau)^{-1/2}$. 
%By setting
%\begin{align*}
%	\lambda_{\tau}^{+}=\lambda_{\tau}^{+}(\xi) := \frac{-1+ \sqrt{1-4\tau |\xi|^2}}{2\tau},
%	\quad
%	\lambda_{\tau}^{-}=\lambda_{\tau}^{-}(\xi) := \frac{-1- \sqrt{1-4\tau |\xi|^2}}{2\tau},
%\end{align*}
%we have
%\begin{align*}
%	\lambda_{\tau}^{+} 
%	&= \frac{-1+ \sqrt{1-4\tau |\xi|^2}}{2\tau}=\frac{2|\xi|^2}{-1-\sqrt{1-4\tau|\xi|^2}}\to -|\xi|^2,
%	\\
%	\lambda_{\tau}^{-} 
%	&= \frac{-1- \sqrt{1-4\tau |\xi|^2}}{2\tau} \to -\infty
%\end{align*}
%and 
%\begin{align*}
%	\frac{\lambda_{\tau}^{\pm}}{ \lambda_{\tau}^{-} -  \lambda_{\tau}^{+}}
%	&= \frac{-1\pm \sqrt{1-4\tau |\xi|^2}}{2\tau} \l(-\frac{\tau}{ \sqrt{1-4\tau |\xi|^2}} \r) 
%	\\
%	&=- \frac{-1\pm \sqrt{1-4\tau |\xi|^2}}{2\sqrt{1-4\tau |\xi|^2}}
%	\to 
%	\begin{cases}
%	0 & \text{ if } \lambda_{\tau}^{+},
%	\\
%	1 & \text{ if } \lambda_{\tau}^{-}.
%	\end{cases}
%\end{align*}
We divide the operators of $L_2$ into three parts as follows. 
\begin{align*}
	\frac{1}{\tau} \widehat{\cD_{\tau}}(t) + \partial_t \widehat{\cD_{\tau}}(t) -\widehat{e^{t\Delta}}
	&=  \left\{\frac{1}{ \lambda_{\tau}^{-} -  \lambda_{\tau}^{+}} 
	\l(\lambda_{\tau}^{-} e^{t \lambda_{\tau}^{+}}
	- \lambda_{\tau}^{+} e^{t \lambda_{\tau}^{-}} \r) -e^{-t|\xi|^2} \right\}
	\cF
	\\
	&= \left(\frac{\lambda_{\tau}^{-} }{ \lambda_{\tau}^{-} -  \lambda_{\tau}^{+}} -1\right)e^{t \lambda_{\tau}^{+}}\cF
	\\
	&\quad + (e^{t \lambda_{\tau}^{+}} -e^{-t|\xi|^2})\cF
	\\
	&\quad - \frac{\lambda_{\tau}^{+} }{ \lambda_{\tau}^{-} -  \lambda_{\tau}^{+}}e^{t \lambda_{\tau}^{-}}\cF
	\\
	&=:I+J+K
\end{align*}
We set $[\xi]_{\tau}:=\sqrt{1-4\tau |\xi|^2}$. Since $|\xi| \leq (8\tau)^{-1/2}$, we have $[ \xi ]_{\tau}\geq  1/\sqrt{2}$. Then, it holds that
\begin{align}
\label{eq222}
	\frac{\lambda_{\tau}^{-} }{ \lambda_{\tau}^{-} -  \lambda_{\tau}^{+}} -1
	=\frac{\lambda_{\tau}^{+} }{ \lambda_{\tau}^{-} -  \lambda_{\tau}^{+}}
	= \frac{2\tau|\xi|^2}{[\xi]_{\tau}(1+[\xi]_{\tau} )} \lesssim \tau^{\frac{1}{2}}|\xi|.
\end{align}
Now, We have the following Strichartz type estimate.
\begin{align*}
	\| e^{t \lambda_{1}^{\pm}(\nabla)}\chi_{\leq1} f \|_{\mathscr{E}^0 \cap \mathscr{H}_2^0} \cleq  \|f\|_{L^2}
\end{align*}
This follows from the argument by e.g. \cite{IIOW19} and \cite{Inu19}, and thus we omit the detail. By the scaling, this estimate implies
\begin{align*}
	\| e^{t \lambda_{ \tau}^{\pm}(\nabla)}\chi_{\leq \tau^{-1/2} } f \|_{\mathscr{E}^0 \cap \mathscr{H}_2^0} 
	\cleq  \|f\|_{L^2},
\end{align*}
where we note that the scaling order $\alpha(q,r)$ of $\tilde{\tau}$ disappears since $\alpha(\mathfrak{h})=\alpha(\mathfrak{e})=0$. 
Thus, we obtain
\begin{align*}
	\|\chi_{<(8\tau)^{-1/2}} I f\|_{\mathscr{E}^0 \cap \mathscr{H}_2^0}
	\lesssim \tau^{\frac{1}{2}} \|f\|_{H^1}.
\end{align*}
By \eqref{eq222}, 
%\begin{align*}
%	\frac{\lambda_{\tau}^{+} }{ \lambda_{\tau}^{-} -  \lambda_{\tau}^{+}}= \frac{4\tau |\xi|^2}{2[\xi]_{\tau} (1+[\xi]_{\tau})} \lesssim \tau^{\frac{1}{2}}|\xi|,
%\end{align*}
we also have
\begin{align*}
	\|\chi_{< (8\tau)^{-1/2} } K f\|_{\mathscr{E}^0 \cap \mathscr{H}_2^0}
	\lesssim \tau^{\frac{1}{2}} \|f\|_{H^1}.
\end{align*}
To estimate $J$, we define $h(\tau):=e^{t\lambda_{\tau}^{+}}$. Then, $J=h(\tau) - h(0)$. Therefore, by the fundamental theorem of calculus, we have
\begin{align*}
	\left\| \chi_{\leq (8\tau)^{-1/2}}(\nabla) J f \right\|_{\mathscr{E}^0 \cap \mathscr{H}_2^0}
	&=\left\| \mathcal{F}^{-1}\{ (h(\tau) - h(0)) \chi_{\leq(8\tau)^{-1/2}} \hat{f} \} \right\|_{\mathscr{E}^0 \cap \mathscr{H}_2^0}
	\\
	&\leq \int_0^\tau \left\| \mathcal{F}^{-1}\{h'(\tilde{\tau}) \chi_{\leq(8\tau)^{-1/2} } \hat{f}\} \right\|_{\mathscr{E}^0 \cap \mathscr{H}_2^0}  d\tilde{\tau}
	\\
	&\leq \int_0^\tau \left\| \mathcal{F}^{-1}\{h'(\tilde{\tau}) \chi_{\leq(8\tilde{\tau})^{-1/2}} \hat{f}\} \right\|_{\mathscr{E}^0 \cap \mathscr{H}_2^0}  d\tilde{\tau}
\end{align*}
where we used $\tilde{\tau} \leq \tau$ in the last inequality. 
%where we note $h(0)=e^{-t|\xi|^2}$ since as $\tau \to 0$ we have
%\begin{align*}
%	\lambda_{\tau}^{+} 
%	&= \frac{-1+ \sqrt{1-4\tau |\xi|^2}}{2\tau}=\frac{2|\xi|^2}{-1-\sqrt{1-4\tau|\xi|^2}}\to -|\xi|^2,
%	\\
%	\lambda_{\tau}^{-} 
%	&= \frac{-1- \sqrt{1-4\tau |\xi|^2}}{2\tau} \to -\infty
%\end{align*}
%and 
%\begin{align*}
%	\frac{\lambda_{\tau}^{\pm}}{ \lambda_{\tau}^{-} -  \lambda_{\tau}^{+}}
%	&= \frac{-1\pm \sqrt{1-4\tau |\xi|^2}}{2\tau} \l(-\frac{\tau}{ \sqrt{1-4\tau |\xi|^2}} \r) 
%	\\
%	&=- \frac{-1\pm \sqrt{1-4\tau |\xi|^2}}{2\sqrt{1-4\tau |\xi|^2}}
%	\to 
%	\begin{cases}
%	0 & \text{ if } \lambda_{\tau}^{+},
%	\\
%	1 & \text{ if } \lambda_{\tau}^{-}.
%	\end{cases}
%\end{align*}
%By the fundamental theorem of calculus, we obtain
%\begin{align*}
%	&\left\| \mathcal{F}^{-1}\{ (h(\tau) - h(0))  \widehat{\chi_{\leq1/4\sqrt{\tau}}} \hat{f} \} \right\|_{\mathscr{E}^0 \cap \mathscr{H}_2^0}
%	\\
%	&=\left\| \mathcal{F}^{-1} \left\{ \int_0^\tau h'(\tilde{\tau}) d\tilde{\tau}  \widehat{\chi_{\leq1/4\sqrt{\tau}}} \hat{f} \right\} \right\|_{\mathscr{E}^0 \cap \mathscr{H}_2^0}
%	\\
%	&\leq \int_0^\tau \left\| \mathcal{F}^{-1}\{h'(\tilde{\tau}) \widehat{\chi_{\leq1/4\sqrt{\tau}}} \hat{f}\} \right\|_{\mathscr{E}^0 \cap \mathscr{H}_2^0}  d\tilde{\tau}
%\end{align*}
By a simple calculation, we have
\begin{align*}
	h'(\tilde{\tau})
	=\frac{-4|\xi|^4}{\langle \xi \rangle_{\tilde{\tau}} (1+\langle \xi \rangle_{\tilde{\tau}})^2} t e^{t\lambda_{\tilde{\tau}}^{+}}
	\lesssim  \tilde{\tau}^{-\frac{1}{2}}|\xi| t|\xi|^2e^{t\lambda_{\tilde{\tau}}^{+}}
	\lesssim  \tilde{\tau}^{-\frac{1}{2}}|\xi|
%	&=
%	\l( \frac{\lambda_{\tau}^{-}}{ \lambda_{\tau}^{-} -  \lambda_{\tau}^{+}}  e^{t \lambda_{\tau}^{+}}
%	-\frac{\lambda_{\tau}^{+}}{ \lambda_{\tau}^{-} -  \lambda_{\tau}^{+}}  e^{t \lambda_{\tau}^{-}} \r)'
%	\\
%	&=|\xi|^2 \alpha^{-3} e^{t \lambda_{\tau}^{+}} +\frac{1+ \alpha}{2\alpha} e^{t \lambda_{\tau}^{+}} \times t (-(\lambda_{\tau}^{+})^2 \alpha^{-1})
%	\\
%	&\quad - \l\{  |\xi|^2 \alpha^{-3} e^{t \lambda_{\tau}^{-}} +\frac{1-\alpha}{2\alpha} e^{t \lambda_{\tau}^{-}} \times t ((\lambda_{\tau}^{-})^2 \alpha^{-1}) \r\}
%	\\
%	&=\alpha^{-3} e^{t \lambda_{\tilde{\tau}}^{+}} \l( |\xi|^2 - \frac{t}{2\tilde{\tau}} |\xi|^2 \alpha (1-\alpha) \r)
%	\\
%	&\quad - \alpha^{-3} e^{t \lambda_{\tilde{\tau}}^{-}} \l( |\xi|^2 +\frac{t}{2\tilde{\tau}} |\xi|^2 \alpha (1+\alpha) \r).
\end{align*}
when $|\xi| \leq (8\tilde{\tau})^{-1/2}$. Thus, by the Strichartz estimates, we obtain
\begin{align*}
	\int_0^\tau \left\| \mathcal{F}^{-1}\{h'(\tilde{\tau})  \chi_{\leq (8\tilde{\tau})^{-1/2} }  \hat{f}\} \right\|_{\mathscr{E}^0 \cap \mathscr{H}_2^0}  d\tilde{\tau}
	&\lesssim \int_{0}^{\tau} \tilde{\tau}^{-\frac{1}{2}} d\tilde{\tau} \|f\|_{H^1}
	\\
	&\approx \tau^{\frac{1}{2}} \|f\|_{H^1}. 
\end{align*}
As a conclusion, we get
\begin{align*}
	&\left\| \left( \frac{1}{\tau} \cD_{\tau}(t) + \partial_t \cD_{\tau}(t) - e^{t\Delta}\right)  \chi_{< (8\tau)^{-1/2}  }  f \right\|_{\mathscr{E}^0 \cap \mathscr{H}_2^0}
	\\
	&\lesssim \| \chi_{< (8\tau)^{-1/2} }  f\|_{L^2}
	\\
	&\lesssim \tau^{\frac{1}{2}} \| f\|_{H^1}.
\end{align*}

Next, we consider the estimate of the nonlinear terms $N_1$ and $N_2$. 

\noindent{\bf The $\mathscr{Y}$-estimate of $N_1$:}

It holds from the Strichartz estimates that
\begin{align*}
	\| N_1 \|_{\mathscr{Y}} \lesssim  T^{1-\theta} \| \mathcal{N}(u_{\tau}) - \mathcal{N}(v) \|_{[\mathscr{E}^0,\mathscr{R}_0^0]_{\theta}+[\mathscr{E}^0,\mathscr{R}_3^0]_{\theta}}.
\end{align*}

Since we have
\begin{align*}
	|\mathcal{N}(u_{\tau}) - \mathcal{N}(v)| \cleq (|u_{\tau}|^{p-1}+ |v|^{p-1})|u_{\tau}-v|, 
\end{align*}
by dividing $u_{\tau},v$ into low frequency part $u_{\tau,l},v_l$ and high frequency part $u_{\tau,h},v_h$, we have
\begin{align*}
	|\mathcal{N}(u_{\tau}) - \mathcal{N}(v)| 
	&\lesssim (|u_{\tau,l}|^{p-1} + |u_{\tau,h}|^{p-1} + |v_l|^{p-1} + |v_h|^{p-1}) |u_{\tau}-v| 
	\\
	&\lesssim (|u_{\tau,l}|^{p-1}+ |v_l|^{p-1}) |u_{\tau}-v| 
	+(|u_{\tau,h}|^{p-1} + |v_h|^{p-1}) |u_{\tau}-v|.
\end{align*}
By Lemma \ref{lem3.6} and the H\"{o}lder inequality, we obtain
\begin{align*}
	&\| \mathcal{N}(u_{\tau}) - \mathcal{N}(v) \|_{[\mathscr{E}^0,\mathscr{R}_0^0]_{\theta}+[\mathscr{E}^0,\mathscr{R}_3^0]_{\theta}}
	\\
	&\lesssim 
	\| (|u_{\tau,l}|^{p-1}+ |v_l|^{p-1}) |u_{\tau}-v| \|_{[\mathscr{E}^0,\mathscr{R}_0^0]_{\theta} }
	+\|(|u_{\tau,h}|^{p-1} + |v_h|^{p-1}) |u_{\tau}-v| \|_{[\mathscr{E}^0,\mathscr{R}_3^0]_{\theta}}
	\\
	&\lesssim (\| u_{\tau,l}\|_{[\mathscr{E}^0,\mathscr{H}_1]_{\theta}}^{{p-1}}  +\| v_l\|_{[\mathscr{E}^0,\mathscr{H}_1]_{\theta}}^{{p-1}} )
	\|u_{\tau}-v\|_{[\mathscr{E}^0,\mathscr{H}_2^0]_{\theta} }
	\\
	&\quad+ (\| u_{\tau,h}\|_{[\mathscr{E}^0,\mathscr{W}_1]_{\theta}}^{{p-1}}  +\| v_h\|_{[\mathscr{E}^0,\mathscr{W}_1]_{\theta}}^{{p-1}} )
	\|u_{\tau}-v \|_{[\mathscr{E}^0,\mathscr{H}_2^0]_{\theta} }
\end{align*}
By the uniform boundedness, we find that $\| u_{\tau,l}\|_{[\mathscr{E}^0,\mathscr{H}_1]_{\theta}}$ and $\| u_{\tau,h}\|_{[\mathscr{E}^0,\mathscr{W}_1]_{\theta}}$ are bounded uniformly in $\tau$. Moreover, $\| v_l\|_{[\mathscr{E}^0,\mathscr{H}_1]_{\theta}}$ and $\| v_h\|_{[\mathscr{E}^0,\mathscr{W}_1]_{\theta}}$ are also bounded. Thus, we have
\begin{align*}
	\| \mathcal{N}(u_{\tau}) - \mathcal{N}(v) \|_{[\mathscr{E}^0,\mathscr{R}_0^0]_{\theta}+[\mathscr{E}^0,\mathscr{R}_3^0]_{\theta}} 
	\lesssim \|u_{\tau}-v \|_{\mathscr{Y}}. 
\end{align*}
This means $\| N_1 \|_{\mathscr{Y}} \cleq T^{1-\theta} \|u_{\tau}-v \|_{\mathscr{Y}}$.

\noindent{\bf The $\mathscr{Y}$-estimate of $N_2$:}
In the same way as the estimate of $L_2$, we devide $N_1$ into the high, middle, and low frequency parts. First, we treat the high frequency part.
\begin{align*}
	&\left\| \int_{0}^{t} \left( \frac{1}{\tau} \mathcal{D}_{\tau}(t-s)  - e^{(t-s)\Delta} \right) \chi_{> (8\tau)^{-1/2} } \mathcal{N}(v) ds \right\|_{\mathscr{Y}}
	\\
	&\lesssim 
	\left\| \int_{0}^{t} \frac{1}{\tau} \mathcal{D}_{\tau}(t-s) \chi_{> (8\tau)^{-1/2} } \mathcal{N}(v) ds \right\|_{\mathscr{Y}}
	 +\left\| \int_{0}^{t}  e^{(t-s)\Delta} \chi_{> (8\tau)^{-1/2} } \mathcal{N}(v) ds \right\|_{\mathscr{Y}}
\end{align*}
By the Strichartz estimate, we have
\begin{align*}
	\left\| \int_{0}^{t} \frac{1}{\tau} \mathcal{D}_{\tau}(t-s) \chi_{>  (8\tau)^{-1/2} } \mathcal{N}(v) ds \right\|_{\mathscr{Y}}
%	\\
%	&\lesssim 
%	\|  \chi_{> 1/4\sqrt{\tau}}\mathcal{N}(v)\|_{[\mathscr{E}^0,\mathscr{R}_0^0]_{\theta}}
%	+\tau^{\theta(\frac{\gamma(\mathfrak{h}_2)}{2}-\frac{1}{2})} \||\nabla|^{\theta(\gamma(\mathfrak{h}_2)-1)} \chi_{> 1/\sqrt{\tau}} \mathcal{N}(v)\|_{[\mathscr{E}^0,\mathscr{R}_0^0]_{\theta}}
%	\\
%	&\quad +\tau^{\theta(\gamma(\mathfrak{h}_2)-\frac{1}{2})} \||\nabla|^{\theta(2\gamma(\mathfrak{h}_2)-1)} \chi_{> 1/\sqrt{\tau}} \mathcal{N}(v)\|_{[\mathscr{E}^0,\mathscr{R}_0^0]_{\theta}}
%	\\
	&\lesssim \tau^{\frac{1}{2}} \||\nabla| \mathcal{N}(v)\|_{[\mathscr{E}^0,\mathscr{R}_0^0]_{\theta}}
	\\
	&\lesssim  \tau^{\frac{1}{2}} \|v\|_{[\mathscr{E}^0,\mathscr{H}_1]_{\theta}}^{p-1}   \| v\|_{[\mathscr{E},\mathscr{H}_2]_{\theta} }
\end{align*}
%where we note that $\delta(\mathfrak{e},\mathfrak{h})=\delta(\mathfrak{h},\mathfrak{e})=0$ and $2\gamma(\mathfrak{h})<1$. 
Moreover, we have
\begin{align*}
	\left\| \int_{0}^{t}  e^{(t-s)\Delta} \chi_{>  (8\tau)^{-1/2} } \mathcal{N}(v) ds \right\|_{\mathscr{Y}}
	&\lesssim \|\chi_{> 1/\sqrt{\tau}} \mathcal{N}(v)\|_{[\mathscr{E}^0,\mathscr{R}_0^0]_{\theta}}
	\\
	&\lesssim \tau^{\frac{1}{2}} \||\nabla| \mathcal{N}(v)\|_{[\mathscr{E}^0,\mathscr{R}_0^0]_{\theta}}
	\\
	&\lesssim  \tau^{\frac{1}{2}} \|v\|_{[\mathscr{E}^0,\mathscr{H}_1]_{\theta}}^{p-1}   \| v\|_{[\mathscr{E},\mathscr{H}_2]_{\theta} }.
\end{align*}
%Next, we consider the middle frequency part. By the Strichartz estimate, we obtain
%\begin{align*}
%	&\left\| \int_{0}^{t} \left( \frac{1}{\tau} \mathcal{D}_{\tau}(t-s)  - e^{(t-s)\Delta} \right) \chi_{ \tau^{-1/4} < \cdot<\tau^{-1/2}} \mathcal{N}(v) ds \right\|_{\mathscr{Y}}
%	\\
%	&\lesssim 
%	\tau^{\frac{1}{4}} \|v\|_{[\mathscr{E}^0,\mathscr{H}_1]_{\theta}}^{p-1}   \| v\|_{[\mathscr{E},\mathscr{H}_2]_{\theta} }.
%\end{align*}

Next, we consider the low frequency part.
%We set
%\begin{align*}
%	k(\tau):=-\alpha^{-1} (-e^{(t-s)\lambda_{\tau}^{+}} + e^{(t-s) \lambda_{\tau}^{-}}),
%\end{align*}
%where we recall $\alpha=\sqrt{1-4\tau|\xi|^2}$. Then, we have $k(0)=e^{-(t-s)|\xi|^2}$. 
The symbol of $\frac{1}{\tau}\mathcal{D}_{\tau}(t) - e^{t\Delta}$ is calculated by
\begin{align*}
	\frac{-1}{\tau(\lambda_{\tau}^{-}- \lambda_{\tau}^{+})} (e^{t\lambda_{\tau}^{+}} - e^{t\lambda_{\tau}^{-}}) - e^{-t|\xi|^2}
	&=\left( \frac{-1}{\tau(\lambda_{\tau}^{-}- \lambda_{\tau}^{+})}-1 \right)e^{t\lambda_{\tau}^{+}}
	\\
	&\quad + e^{t\lambda_{\tau}^{+}} - e^{-t|\xi|^2}
	\\
	&\quad + \frac{1}{\tau(\lambda_{\tau}^{-}- \lambda_{\tau}^{+})}  e^{t\lambda_{\tau}^{-}}
	\\
	&=:H+J+M
\end{align*}
Since we have
\begin{align*}
	& \frac{-1}{\tau(\lambda_{\tau}^{-}- \lambda_{\tau}^{+})}-1 = \frac{4\tau|\xi|^2}{[\xi]_{\tau}(1+[\xi]_{\tau})} \lesssim \tau^{\frac{1}{2}}|\xi|
\end{align*}
when $|\xi| \leq (8\tau)^{-1/2}$, we can estimate $H$ by 
\begin{align*}
	\left\| \int_{0}^{t} H \chi_{\leq (8\tau)^{-1/2}} \mathcal{N}(v) ds \right\|_{\mathscr{Y}}
	\lesssim \tau^{\frac{1}{2}}  \|v\|_{[\mathscr{E}^0,\mathscr{H}_1]_{\theta}}^{p-1}   \| v\|_{[\mathscr{E},\mathscr{H}_2]_{\theta} }
\end{align*}
%and 
%\begin{align*}
%	\left\| \int_{0}^{t} M_1 \chi_{\leq (8\tau)^{-1/2}} \mathcal{N}(v) ds \right\|_{\mathscr{Y}}
%	\lesssim \tau^{\frac{1}{2}}  \|v\|_{[\mathscr{E}^0,\mathscr{H}_1]_{\theta}}^{p-1}   \| v\|_{[\mathscr{E},\mathscr{H}_2]_{\theta} }
%\end{align*}
where we used the inhomogeneous Strichartz estimate for $e^{t\lambda_{\tau}^{\pm}}$. 
In the similar way to the estimate of $L_2$, we obtain
\begin{align*}
	\left\| \int_{0}^{t} J \chi_{\leq (8\tau)^{-1/2}} \mathcal{N}(v) ds \right\|_{\mathscr{Y}}
	\lesssim \tau^{\frac{1}{2}} T \|v\|_{[\mathscr{E}^0,\mathscr{H}_1]_{\theta}}^{p-1}   \| v\|_{[\mathscr{E},\mathscr{H}_2]_{\theta} }.
\end{align*}
We estimate $M$. Since it holds
\begin{align*}
	 \left| \frac{1}{\tau(\lambda_{\tau}^{-}- \lambda_{\tau}^{+})}\right| \lesssim 1,
\end{align*}
it is enough to estimate
\begin{align*}
	\left\| \int_{0}^{t} e^{t\lambda_{\tau}^{-}} \chi_{\leq (8\tau)^{-1/2}} \mathcal{N}(v) ds \right\|_{ \mathscr{Y}}.
\end{align*}
We denote $(q,r)=(\infty,2)$ or $(q,r)= (2(d+2)/d,2(d+2)/d)$. When $1<p\leq d/(d-2)$, we have
\begin{align*}
	\left\| \int_{0}^{t} e^{t\lambda_{\tau}^{-}} \chi_{\leq (8\tau)^{-1/2}} \mathcal{N}(v) ds \right\|_{L_t^qL_x^r}
	&\lesssim  \left\| \int_{0}^{t} e^{-\frac{t-s}{2\tau}} \left\|  \chi_{\leq (8\tau)^{-1/2}} \mathcal{N}(v)\right\|_{L_x^r} ds \right\|_{L_t^q}
	\\
	&\lesssim \tau^{-\frac{d}{2}\left(\frac{1}{2} - \frac{1}{r}\right)} \left\| \int_{0}^{t} e^{-\frac{t-s}{2\tau}} \left\| \mathcal{N}(v)\right\|_{L_x^2} ds \right\|_{L_t^q}
	\\
	&\lesssim \tau^{-\frac{d}{2}\left(\frac{1}{2} - \frac{1}{r}\right)} \left\| \int_{0}^{t} e^{-\frac{t-s}{2\tau}}ds \right\|_{L_t^q}
	 \|v \|_{L_t^\infty H^1}^{p}
	 \\
	&\lesssim \tau^{1-\frac{d}{2}\left(\frac{1}{2} - \frac{1}{r}\right)}T^{\kappa} \|v \|_{L_t^\infty H^1}^{p}
\end{align*}
for a positive constant $\kappa$, where we used the Bernstein inequality and the Sobolev inequality. Since $r=2$ or $r=2(d+2)/d$, we obtain
\begin{align*}
	\tau^{1-\frac{d}{2}\left(\frac{1}{2} - \frac{1}{r}\right)} \lesssim \tau^{\frac{1}{2}}.
\end{align*}
When $d/(d-2) < p < (d+2)/(d-2)$, we set $1/r_0= 1/r - 1/d + (d-2)(p-1)/(2d)$. Then, we have $r > r_0 >1$ and, by the Sobolev inequality, we have $\left\| \mathcal{N}(v)\right\|_{L_x^{r_0}} \lesssim \|v\|_{H^1}^{p-1} \|v\|_{W^{1,r}}$.
\begin{align*}
	\left\| \int_{0}^{t} e^{t\lambda_{\tau}^{-}} \chi_{\leq (8\tau)^{-1/2}} \mathcal{N}(v) ds \right\|_{L_t^qL_x^r}
	&\lesssim  \left\| \int_{0}^{t} e^{-\frac{t-s}{2\tau}} \left\|  \chi_{\leq (8\tau)^{-1/2}} \mathcal{N}(v)\right\|_{L_x^r} ds \right\|_{L_t^q}
	\\
	&\lesssim \tau^{-\frac{d}{2}\left(\frac{1}{r_0} - \frac{1}{r}\right)} \left\| \int_{0}^{t} e^{-\frac{t-s}{2\tau}} \left\| \mathcal{N}(v)\right\|_{L_x^{r_0}} ds \right\|_{L_t^q}
	\\
	&\lesssim \tau^{-\frac{d}{2}\left(\frac{1}{r_0} - \frac{1}{r}\right)} \left\| \int_{0}^{t} e^{-\frac{t-s}{2\tau}} \|v(s)\|_{W^{1,r}}ds \right\|_{L_t^q}
	 \|v \|_{L_t^\infty H^1}^{p-1} 
	 \\
	&\lesssim \tau^{1-\frac{d}{2}\left(\frac{1}{r_0} - \frac{1}{r}\right)} \|v\|_{L_t^q W^{1,r}} \|v \|_{L_t^\infty H^1}^{p-1}
\end{align*}
where we used the Young inequality in the last. Since $p<1+4/(d-2)$, we have
\begin{align*}
	1-\frac{d}{2}\left(\frac{1}{r_0} - \frac{1}{r}\right) >\frac{1}{2}.  
\end{align*}
Therefore, combining these estimates, we have
\begin{align*}
	\left\| \int_{0}^{t} M \chi_{\leq (8\tau)^{-1/2}} \mathcal{N}(v) ds \right\|_{\mathscr{Y}}
	\lesssim \tau^{\frac{1}{2}}  T^{\kappa_0} \|v \|_{L_t^\infty H^1}^{p-1} \|v\|_{[\mathscr{E},\mathscr{H}_2]_{\theta}}.
\end{align*}
for a positive constant $\kappa_0$. Since $\|v \|_{L_t^\infty H^1}$ and $\|v\|_{[\mathscr{E},\mathscr{H}_2]_{\theta}}$ are bounded on $[0,T]$, we obtain $O(\tau^{1/2})$.

Combining these estimates, we obtain
\begin{align*}
	\|N_2\|_{\mathscr{Y}} 
	\lesssim \tau^{\frac{1}{2}} %\|v\|_{[\mathscr{E}^0,\mathscr{H}_1]_{\theta}}^{p-1}   \| v\|_{[\mathscr{E},\mathscr{H}_2]_{\theta} }.
\end{align*}
%Since $v$ is bounded, we haev
%\begin{align*}
%	\|N_2\|_{\mathscr{Y}} \lesssim \tau^{\frac{1}{4}}. 
%\end{align*}

\noindent{\bf Conclusion of the $\mathscr{Y}$-estimate:} 

Thus, we have
\begin{align*}
	\|u_{\tau} - v\|_{\mathscr{Y}[0,T]} 
	\lesssim \|f_{\tau}-f\|_{L^2} + \tau \|g_{\tau}\|_{L^2}  + T^{1-\theta} \|u_{\tau}- v\|_{\mathscr{Y}[0,T]} + \tau^{\frac{1}{2}}
\end{align*}
Taking small $T$, we obtain
\begin{align*}
	\|u_{\tau} - v\|_{\mathscr{Y}[0,T]} 
	\lesssim \|f_{\tau}-f\|_{L^2} + \tau \|g_{\tau}\|_{L^2}  + \tau^{\frac{1}{2}}
\end{align*}
Repeating this, we obtain the estimate for any $T<T^{*}$. 
Therefore, we complete the proof of Theorem \ref{thm1.2}.

\subsection{Compactness method for $H^1$-convergence}

We prove $H^1$-convergence by a compactness method.

We show the following statement. 
\begin{proposition}
\label{prop2.8}
%Let the assumption of Theorem \ref{thm1.1} be satisfied. 
Let $f\in H^1(\mathbb{R}^d)$, $v$ be a solution to \eqref{NLH}, and $T_{\max}$ is the maximal existence time of the solution $v$ to \eqref{NLH}. 
If the initial data $(f_{\tau},g_{\tau}) \in  H^1(\mathbb{R}^d) \times  L^2(\mathbb{R}^d)$ satisfies
\begin{align*}
	(f_{\tau},\tau g_{\tau}) \to (f,0) \text{ in } H^1(\mathbb{R}^d)\times L^2(\mathbb{R}^d) \text{ as } \tau \to 0,
\end{align*}
then we have
\begin{align*}
	\| u_{\tau} - v\|_{L^{\infty}(0,T: H^1(\mathbb{R}^d))} 
	+ \tau^{\frac{1}{2}} \| \partial_t u_{\tau}\|_{L^{\infty}(0,T: L^2(\mathbb{R}^d))}
	\to 0
\end{align*}
as $\tau \to 0$ for any $T<\min\{ T^{*},T_{\max}\}$ . 
\end{proposition}

Once we obtain the above proposition, we can show $T^{*}\geq T_{\max}$ as follows. Suppose that $T^{*} <T_{\max}$. Then by the proposition, it holds for arbitrary sufficiently small $\tau$ that
\begin{align*}
	\| u_{\tau} - v\|_{L^{\infty}(0,T: H^1(\mathbb{R}^d))} 
	+ \tau^{\frac{1}{2}} \| \partial_t u_{\tau}\|_{L^{\infty}(0,T: L^2(\mathbb{R}^d))}
	\lesssim 1
\end{align*}
for any $T<T^{*}$. Therefore, we have
\begin{align*}
	\| u_{\tau} \|_{L^{\infty}(0,T: H^1(\mathbb{R}^d))}+ \tau^{\frac{1}{2}} \| \partial_t u_{\tau}\|_{L^{\infty}(0,T: L^2(\mathbb{R}^d))}
	 \lesssim \|v\|_{L^{\infty}(0,T: H^1(\mathbb{R}^d))} +1
\end{align*}
Since $T^{*} <T_{\max}$, we have $\|v\|_{L^{\infty}(0,T: H^1(\mathbb{R}^d))}+ \tau^{1/2} \| \partial_t u_{\tau}\|_{L^{\infty}(0,T: L^2(\mathbb{R}^d))}< C$. Thus, $\| u_{\tau} \|_{L^{\infty}(0,T: H^1(\mathbb{R}^d))}<C$ for any $T<T^{*}$. This and the blow-up alternative implies that we obtain the solution on $[0,T^{*}+\delta)$ for some $\delta>0$. Since the existence time depends only on the norm, we find that $\delta$ is independent of $\tau$. We reach contradiction.

\begin{proof}[Proof of Proposition {\ref{prop2.8}}]
To show this, we use a compactness argument. Let $\varepsilon>0$ be fixed arbitrarily. If there exists $R_{\varepsilon}>0$ independent of $\tau$ such that $\|\chi_{>R_{\varepsilon}} u_{\tau}\|_{\dot{\mathscr{E}}} \lesssim  \varepsilon + \|f_{\tau}-f\|_{H^1}  +\tau^{1/2} \| g_{\tau}\|_{L^2}$, then by taking $R>R_{\varepsilon}$ such that $\| \chi_{>R} v \|_{\dot{\mathscr{E}}}<\varepsilon$ we have
\begin{align*}
	\| u_{\tau} -v \|_{\mathscr{E}}
	&\leq \| u_{\tau} -v \|_{\mathscr{E}^{0}} + \|  u_{\tau} -v \|_{\dot{\mathscr{E}}}
	\\
	&\leq \| u_{\tau} -v \|_{\mathscr{E}^{0}}+ \| \chi_{\leq R} (u_{\tau} -v) \|_{\dot{\mathscr{E}}} 
	+ \| \chi_{>R}(u_{\tau} -v )\|_{\dot{\mathscr{E}}}
	\\
	&\lesssim  (1+R)\|  u_{\tau} -v \|_{\mathscr{E}^{0}} + 2\varepsilon 
	+ \|f_{\tau}-f\|_{H^1}  +\tau^{\frac{1}{2}} \| g_{\tau}\|_{L^2}
\end{align*} 
By the $L^2$-convergence, we have $\|  u_{\tau} -v \|_{\mathscr{E}^{0}} \to 0$ as $\tau \to 0$. Thus, by the assumption on the initial data, we obtain
\begin{align*}
	\lim_{\tau \to 0} \| u_{\tau} -v \|_{\mathscr{E}}=0.
\end{align*}
By the above argument, it is enough to show that there exists $R_{\varepsilon}>0$ independent of $\tau$ such that $\|\chi_{>R_{\varepsilon}} u_{\tau}\|_{\dot{\mathscr{E}}} \lesssim  \varepsilon + \|f_{\tau}-f\|_{H^1}  +\tau^{1/2} \| g_{\tau}\|_{L^2}$. 
%For the linear part of $u_{\tau}$, it holds from the homogeneous Strichartz estimate that
%\begin{align*}
%	\|\chi_{>R_{\varepsilon}} u_{\tau}^{lin}\|_{\mathscr{E}} 
%	&\lesssim \| \chi_{>R_{\varepsilon}} f_{\tau}\|_{H^1}  +\tau \|\chi_{>R_{\varepsilon}} g_{\tau}\|_{L^2}
%	\\
%	&\lesssim \| \chi_{>R_{\varepsilon}} f\|_{H^1} + \|f_{\tau}-f\|_{H^1}  +\tau \| g_{\tau}\|_{L^2}.
%\end{align*}
%Next, we consider the nonlinear part.

Let $R>0$ and $J \in \mathbb{N}$ satisfy $R>2^{J}$. By the Littlewood--Paley decomposition, we have
 \begin{align*}
 	\| \chi_{>R} u_{\tau} \|_{\dot{\mathscr{E}}}
	\lesssim  \| \| P_j \chi_{>R} u_{\tau} \|_{\dot{\mathscr{E}}} \|_{l_j^2}
	\lesssim \| \| P_j u_{\tau} \|_{\dot{\mathscr{E}}} \|_{l^2(j>J)}.
\end{align*}
Therefore, it is enough to show that for any $\varepsilon>0$ there exists $J\in \mathbb{N}$ independent of $\tau$ such tha $\| \| P_j u_{\tau} \|_{\dot{\mathscr{E}}} \|_{l^2(j>J)}<\varepsilon + \|f_{\tau}-f\|_{H^1}  +\tau^{1/2} \| g_{\tau}\|_{L^2}$. By the Strichartz estimate, we obtain
\begin{align*}
	\| P_j u_{\tau} \|_{\dot{\mathscr{E}}\cap \dot{\mathscr{Z}}_\theta \cap \mathscr{X} }
	&\lesssim \|P_j f_{\tau}\|_{H^1} + \tau^{\frac{1}{2}} \|P_j g_{\tau}\|_{L^2} 
	+ \left\| \int_{0}^{t} \frac{1}{\tau} \mathcal{D}_{\tau} (t-s) P_j \mathcal{N}(u_{\tau}(s))ds  \right\|_{\dot{\mathscr{E}}\cap \dot{\mathscr{Z}}_\theta \cap \mathscr{X}}
	\\
	&\lesssim \|P_j f_{\tau}\|_{H^1} + \tau^{\frac{1}{2}} \|P_j g_{\tau}\|_{L^2} 
	+T^{1-\theta} \| P_j \mathcal{N}(u_{\tau}) \|_{\sum_{i=0}^{3}[\dot{\mathscr{E}}, \dot{\mathscr{R}_{i}}]_{\theta}}.
\end{align*}
By \cite[Lemma]{MaNa02}, we have the following estimate.
\begin{align*}
\| P_j \mathcal{N}(u_{\tau}) \|_{\sum_{i=0}^{3}[\dot{\mathscr{E}}, \dot{\mathscr{R}_{i}}]_{\theta}}
\lesssim 2^{-\frac{|j|}{4}} *_j \| u_{\tau} \|_{\mathscr{X}}^{p-1} \| P_j u_{\tau} \|_{\dot{\mathscr{Z}}_\theta},
\end{align*}
where $*_j$ denotes the convolution over $\mathbb{Z}$.
Thus, we have
\begin{align}
\label{eq3.4}
	\| P_j u_{\tau} \|_{\dot{\mathscr{E}}\cap \dot{\mathscr{Z}}_\theta \cap \mathscr{X} }
	\lesssim \|P_j f_{\tau}\|_{H^1} + \tau^{\frac{1}{2}} \|P_j g_{\tau}\|_{L^2}
	+ T^{1-\theta} 2^{-\frac{|j|}{4}} *_j \| P_j u_{\tau} \|_{\dot{\mathscr{Z}}_\theta}
\end{align}
since $\| u_{\tau} \|_{\mathscr{X}}^{p-1}$ is uniformly bounded. 
By convoluting with $2^{-\frac{|j|}{5}}$, we have
\begin{align*}
	&2^{-\frac{|j|}{5}}*_j\| P_j u_{\tau} \|_{\dot{\mathscr{E}}\cap \dot{\mathscr{Z}}_\theta \cap \mathscr{X} }
	\\
	&\lesssim 2^{-\frac{|j|}{5}}*_j( \|P_j f_{\tau}\|_{H^1} + \tau^{\frac{1}{2}} \|P_j g_{\tau}\|_{L^2} )
	+ T^{1-\theta} 2^{-\frac{|j|}{5}} *_j \| P_j u_{\tau} \|_{\dot{\mathscr{Z}}_\theta}
\end{align*}
since $2^{-\frac{|j|}{5}} *_j 2^{-\frac{|j|}{4}} \lesssim 2^{-\frac{|j|}{5}}$. If $T$ is sufficiently small, we obtain
\begin{align}
\label{eq3.5}
	2^{-\frac{|j|}{5}}*_j\| P_j u_{\tau} \|_{\dot{\mathscr{E}}\cap \dot{\mathscr{Z}}_\theta \cap \mathscr{X} }
	\lesssim 2^{-\frac{|j|}{5}}*_j( \|P_j f_{\tau}\|_{H^1} + \tau^{\frac{1}{2}} \|P_j g_{\tau}\|_{L^2} ).
\end{align}
It follows from $2^{-\frac{|j|}{4}}<2^{-\frac{|j|}{5}}$ and substituting \eqref{eq3.5} into \eqref{eq3.4} that
\begin{align*}
	\| P_j u_{\tau} \|_{\dot{\mathscr{E}}}
	&\lesssim \|P_j f_{\tau}\|_{H^1} + \tau^{\frac{1}{2}} \|P_j g_{\tau}\|_{L^2}
	\\
	&\quad + T^{1-\theta}  2^{-\frac{|j|}{5}}*_j( \|P_j f_{\tau}\|_{H^1} + \tau \|P_j g_{\tau}\|_{L^2} ).
\end{align*}
Taking $l^2$-norm for $j>J$, by $2^{-\frac{|j|}{5}} \in l^1$ and the Young inequality, we have
\begin{align*}
	\| \| P_j u_{\tau} \|_{\dot{\mathscr{E}}} \|_{l^2(j>J)}
	&\lesssim  \| \|P_j f_{\tau}\|_{H^1}\|_{l^2(j>J)} + \tau^{\frac{1}{2}} \| \|P_j g_{\tau}\|_{L^2}\|_{l^2(j>J)}
\end{align*}
By the Littlewood--Paley decomposition, we have
\begin{align*}
	 \| \|P_j f_{\tau}\|_{H^1}\|_{l^2(j>J)}
	 &\leq  \| \|P_j (f_{\tau}-f)\|_{H^1}\|_{l^2(j>J)} +  \| \|P_j f\|_{H^1}\|_{l^2(j>J)}
	 \\
	 &\lesssim \| f_{\tau}-f \|_{H^1} + \|\chi_{>2^{J-2}} f \|_{H^1}
\end{align*}
and 
\begin{align*}
	 \tau^{\frac{1}{2}} \| \|P_j g_{\tau}\|_{L^2}\|_{l^2(j>J)} \lesssim \tau^{\frac{1}{2}} \|g_{\tau}\|_{L^2}.
\end{align*}
Thus, for any $\varepsilon>0$ there exists $J\in \mathbb{N}$ independent of $\tau$ such that
\begin{align*}
	\| \|P_j f_{\tau}\|_{H^1}\|_{l^2(j>J)}
	\lesssim \varepsilon +\| f_{\tau}-f \|_{H^1} + \tau^{\frac{1}{2}} \|g_{\tau}\|_{L^2}
\end{align*}
since $ \|\chi_{>2^{J-2}} f \|_{H^1} < \varepsilon$ for large $J$. We obtain the statement for small $T$. Repeating this argument, we obtain the $H^1$-convergence of the solution. At last, we show $\tau^{1/2}\|\partial_t u_{\tau}\|_{\mathscr{E}^{0}[0,T]}\to 0$ as $\tau \to 0$. First we show that
\begin{align*}
	\tau^{\frac{1}{2}} \|\chi_{\leq R} \partial_t u_{\tau}\|_{\mathscr{E}^0[0,T]} \to 0
\end{align*}
as $\tau \to 0$ for arbitrary fixed $R>0$. By the Strichartz estimates, we have
\begin{align*}
	&\tau^{\frac{1}{2}}\left\| \left(\frac{1}{\tau} \partial_t \mathcal{D}_{\tau} +  \partial_t^2 \mathcal{D}_{\tau} \right) \chi_{\leq R} f_{\tau}\right\|_{\mathscr{E}^0[0,T]}
	\\
	&\lesssim \tau^{\frac{1}{2}} R \|f_{\tau} - f\|_{H^1} + \tau^{\frac{1}{2}} \left\| \left(\frac{1}{\tau} \partial_t \mathcal{D}_{\tau} +  \partial_t^2 \mathcal{D}_{\tau} \right) \chi_{\leq R} f\right\|_{\mathscr{E}^0[0,T]}
	\\
	&\lesssim  \tau^{\frac{1}{2}} R \|f_{\tau} - f\|_{H^1} + \tau^{\frac{1}{2}} R \| f\|_{H^1}
\end{align*}
and 
\begin{align*}
	\|\partial_t \mathcal{D}_{\tau} \chi_{\leq R} g_{\tau}\|_{\mathscr{E}^0[0,T]} 
	\lesssim \tau^{\frac{1}{2}} \|g_{\tau}\|_{L^2}.
\end{align*}
We also have the following estimate for the inhomogeneous term.
\begin{align*}
	&\left\| \int_0^t \frac{1}{\tau} \partial_t \cD_{\tau}(t-s) \chi_{\leq R} \mathcal{N}(u_{\tau}(s)) ds \right\|_{L_t^{\infty}L_x^2} 
	\\
	&\lesssim \int_{0}^{t} \||\xi|^2 e^{(t-s)\lambda_{\tau}^{+}}\chi_{\leq R} \mathcal{N}(u_{\tau}(s)) \|_{L^2} ds
	\\
	&\quad +\int_{0}^{t} \|\tau^{-1} e^{(t-s)\lambda_{\tau}^{-}}\chi_{\leq R} \mathcal{N}(u_{\tau}(s)) \|_{L^2} ds
\end{align*}
The second term of the right hand side is calculated as follows. 
\begin{align*}
	\int_{0}^{t} \|\tau^{-1} e^{(t-s)\lambda_{\tau}^{-}}\chi_{\leq R} \mathcal{N}(u_{\tau}(s)) \|_{L^2} ds
	&\lesssim \tau^{-1} \int_{0}^{t}  e^{-\frac{t-s}{\tau}} ds  \|\chi_{\leq R} \mathcal{N}(u_{\tau}) \|_{L_t^{\infty} L_x^2}
	\\
	&\lesssim  \|\chi_{\leq R} \mathcal{N}(u_{\tau}) \|_{L_t^{\infty} L_x^2}
%	\int_{0}^{t} \||\xi|^2 e^{(t-s)\lambda_{\tau}^{+}}\chi_{\leq R} \mathcal{N}(u_{\tau}(s)) \|_{L^2} ds
%	\lesssim R^2 T \| \chi_{\leq R} \mathcal{N}(u_{\tau}) \|_{L_t^\infty L_x^2}
\end{align*}
By the Bernstein inequality and the Sobolev inequality, we have
\begin{align*}
	 \| \chi_{\leq R} \mathcal{N}(u_{\tau}(s)) \|_{L^2} 
	& \lesssim R^{d \left(\frac{1}{r_0} - \frac{1}{2}\right)} \| \chi_{\leq R} \mathcal{N}(u_{\tau}(s)) \|_{L^{r_0}}
	\\
	& \lesssim  R^{d\left(\frac{1}{r_0} - \frac{1}{2}\right)} \| u_{\tau} \|_{L^{r_0p}}^{p}
	\\
	&  \lesssim  R^{d\left(\frac{1}{r_0} - \frac{1}{2}\right)} \| u_{\tau} \|_{H^1}^{p}
\end{align*}
where we set
\begin{align*}
	\frac{1}{r_0} =
	\begin{cases}
	\frac{1}{2} & \text{ if } p\leq \frac{d}{d-2},
	\\
	\frac{(d-2)p}{2d}& \text{ if } p> \frac{d}{d-2}.
	\end{cases}
\end{align*}
%By the Young inequality in time variable, we obtain
%\begin{align*}
%	 \tau^{-\delta} \int_{0}^{t} \frac{1}{(t-s)^{1-\delta}} \| \chi_{\leq \tau^{-1/2}} \mathcal{N}(u_{\tau}(s)) \|_{L^2} ds
%	 \lesssim  \tau^{-\delta -\frac{d}{2}\left(\frac{1}{r_0} - \frac{1}{2}\right)} T^{\delta} \| u_{\tau} \|_{\mathscr{E}[0,T]}^{p}
%\end{align*}
Since the first term can be estimated by $|\xi|^2 \leq R^2$ and the similar argument, we obtain
\begin{align*}
	\tau^{\frac{1}{2}}\left\| \int_0^t \frac{1}{\tau} \partial_t \cD_{\tau}(t-s) \chi_{\leq R} \mathcal{N}(u_{\tau}(s)) ds \right\|_{L_t^{\infty}L_x^2}
	\lesssim \tau^{\frac{1}{2}} \left(R^2 T + R^{d\left(\frac{1}{r_0} - \frac{1}{2}\right)}\right) \|u_{\tau}\|_{H^1}^{p}.
\end{align*}
Combining the above estimates, we obtain $\tau^{1/2} \|\chi_{\leq R} \partial_t u_{\tau}\|_{\mathscr{E}^0[0,T]} \to 0$ as $\tau \to 0$.

Next,  in order to find $R_{\varepsilon}>0$ independent of $\tau$ such that $\tau^{1/2} \|\chi_{> R_{\varepsilon}} \partial_t u_{\tau}\|_{\mathscr{E}^0[0,T]} < \varepsilon + \|f_{\tau} - f\|_{H^1} + \tau^{1/2} \|g_{\tau}\|$, it is enough to do the similar argument as above for $\|\chi_{>R_{\varepsilon}} u_{\tau}\|_{\dot{\mathscr{E}}}$. 
%
%as $\tau \to 0$ since we have \eqref{eq3.0.1}. It holds by the triangle inequality that
%\begin{align*}
%	\tau^{\frac{1}{2}}\left\| \left(\frac{1}{\tau} \partial_t \mathcal{D}_{\tau} +  \partial_t^2 \mathcal{D}_{\tau} \right) f_{\tau}\right\|_{\mathscr{E}^0}
%	\leq \|f_{\tau} -f\|_{H^1} + \tau^{\frac{1}{2}}\left\| \left(\frac{1}{\tau} \partial_t \mathcal{D}_{\tau} +  \partial_t^2 \mathcal{D}_{\tau} \right) f\right\|_{\mathscr{E}^0}.
%\end{align*}
%For any $\varepsilon>0$, there exists $R>0$ such that $\|\chi_{>R}f\|_{H^1}<\varepsilon$. We have
%\begin{align*}
%	\tau^{\frac{1}{2}}\left\| \left(\frac{1}{\tau} \partial_t \mathcal{D}_{\tau}(t) +  \partial_t^2 \mathcal{D}_{\tau}(t) \right) \chi_{\leq R} f\right\|_{L^2}
%	\lesssim \tau^{\frac{1}{2}} \||\xi|^2 \chi_{\leq R}  f\|
%	\lesssim \tau^{\frac{1}{2}} R \|f\|_{H^1}
%\end{align*}
%and 
%\begin{align*}
%	&\tau^{\frac{1}{2}}\left\| \left(\frac{1}{\tau} \partial_t \mathcal{D}_{\tau}(t) +  \partial_t^2 \mathcal{D}_{\tau}(t) \right) \chi_{> R} f\right\|_{L^2}
%	\\
%	&\lesssim \tau^{\frac{1}{2}} \||\xi|^2 \chi_{> R} \chi_{\leq \tau^{-1/2}}  f\| + \||\xi| \chi_{> R} \chi_{>\tau^{-1/2}}  f\| 
%	\\
%	&\lesssim \|\chi_{> R} f\|_{H^1} <\varepsilon
%\end{align*}
%for small $\tau>0$. This implies that
%\begin{align*}
%	\lim_{\tau \to 0 } \left(\tau^{\frac{1}{2}}\left\| \left(\frac{1}{\tau} \partial_t \mathcal{D}_{\tau} +  \partial_t^2 \mathcal{D}_{\tau} \right) f\right\|_{\mathscr{E}^0}\right)=0.
%\end{align*}
%Since $f_{\tau}$ converges to $f$ in $H^1$, we obtain the $L^2$-convergence of $\partial_t u_{\tau}$. 
We finish the proof.
\end{proof}

%%%%%%%%%%%%%%%%%%%%%%%%%%%%%%%%%%%%%%%%%%%%%%%%%%%%%%%%

\subsection{The case of $d=1,2$}
\label{sec2.5}
We give the proofs in the cases of $d=1,2$. 

\subsubsection{The case of $d=1$.} 
By the Sobolev embedding $L^\infty \supset H^1$, we have 
\begin{align*}
	\|u_{\tau}\|_{L^\infty H^1}
	&\lesssim \|f_{\tau}\|_{H^1} +\tau^{\frac{1}{2}} \|g_{\tau}\|_{L^2} + T \| \mathcal{N}(u_{\tau}) \|_{L^\infty H^1}
	\\
	&\lesssim \|f_{\tau}\|_{H^1} +\tau^{\frac{1}{2}} \|g_{\tau}\|_{L^2} + T\|u_{\tau}\|_{L^\infty L^\infty}^{p-1} \| u_{\tau} \|_{L^\infty H^1}
	\\
	&\lesssim \|f_{\tau}\|_{H^1} +\tau^{\frac{1}{2}} \|g_{\tau}\|_{L^2} + T \| u_{\tau} \|_{L^\infty H^1}^{p}
\end{align*}
Therefore, we obtain the uniform boundedness. 
We can calculate $\|u_{\tau} -v\|_{L^\infty L^2}$ in the similar way to the case of $d \geq 3$. Since we do not need to calculate $\mathscr{H}_2^0$-norm, the difference are estimated much easier than the case of $d \geq 3$. The $H^1$-convergence is directly shown without a compactness method. We use the Parseval equality and the convergence of the symbols and then apply the Lebesgue dominated convergence theorem. 

\subsubsection{The case of $d=2$.} 
By the Strichartz estimates, we obtain
\begin{align*}
	\|u_{\tau}\|_{L^\infty H^1}
	&\lesssim \|f_{\tau}\|_{H^1} +\tau^{\frac{1}{2}} \|g_{\tau}\|_{L^2} +  \| \mathcal{N}(u_{\tau}) \|_{L^{\tilde{q}'} B_{\tilde{r}',2}^{1}},
\end{align*}
where $1/\tilde{q}=1/2-1/\tilde{r}$, $2<\tilde{r}<2/(2-p)$ if $1<p<2$, and $2<\tilde{r}<\infty$ if $p\geq 2$. 
Then, by the H\"{o}lder inequality, we have
\begin{align*}
	\| \mathcal{N}(u) \|_{L^{\tilde{q}'} B_{\tilde{r}',2}^{1}}
	\lesssim \|u_{\tau}\|_{L^{\tilde{q}'(p-1)} L^{r (p-1)}}^{p-1} \| u_{\tau} \|_{L^\infty H^1},
\end{align*}
where $1/r=1/2-1/\tilde{r}$. Since $2<\tilde{r}<2/(2-p)$ if $1<p<2$, and $2<\tilde{r}<\infty$ if $p\geq 2$, we have $r \in (2,\infty)$. Thus, by the Sobolev embedding $L^r \supset H^1$, we obtain
\begin{align*}
	\| \mathcal{N}(u) \|_{L^{\tilde{q}'} B_{\tilde{r}',2}^{1}}
	&\lesssim \|u_{\tau}\|_{L^{\tilde{q}'(p-1)}H^1}^{p-1} \| u_{\tau} \|_{L^\infty H^1}
	\\
	&\lesssim T^{1-\frac{1}{\tilde{q}'}} \|u_{\tau}\|_{L^{\infty}H^1}^{p},
\end{align*}
where $1-1/\tilde{q}'>0$. Therefore, we obtain the uniform boundedness.
%\begin{align*}
%	&\lesssim \|f_{\tau}\|_{H^1} + \|g_{\tau}\|_{L^2} 
%	+\||u_{\tau}|^{p-1}\|_{L^1  B_{\infty,2}^{0}} \| u_{\tau} \|_{L^\infty H^1}
%	\\
%	&\lesssim \|f_{\tau}\|_{H^1} + \|g_{\tau}\|_{L^2} 
%	+\|u_{\tau}\|_{L^{p-1}  B_{\infty,2}^{0}}^{p-1} \| u_{\tau} \|_{L^\infty H^1}
%	\\
%	&\lesssim \|f_{\tau}\|_{H^1} + \|g_{\tau}\|_{L^2} 
%	+T^{1-\frac{p-1}{2}} \|u_{\tau}\|_{L^{2}  B_{\infty,2}^{0}}^{p-1} \| u_{\tau} \|_{L^\infty H^1}
%\end{align*}
If $T$ is sufficiently small, we obtain the uniform boundedness. For the estimate of the difference, we use $L^2$-norm.
The $L^2$-convergence and $H^1$-convergence can be shown in the same method as in the case of $d=1$. 

%%%%%%%%%%%%%%%%%%%%%%%%%%%%%%%%%%%%%%%%%%%%%%%%%%%%%%%%

\section{Global $\dot{H}^1$-convergence}

In this section, we prove Theorems \ref{thm1.4} and \ref{cor1.5}. 

In the case of $d \geq 3$, we set 
\begin{align*}
	\mathscr{V}_{\eta}&:=[\mathscr{H}_2^0,\mathscr{H}_1]_{\eta} | [\mathscr{H}_2^0,\mathscr{W}_1]_{\eta}
	\\
	\mathscr{Z}&:=\mathscr{Z}_1 = \mathscr{H}_2 | \mathscr{W}_2 \cap \mathscr{D}_1
\end{align*}
where $p=\eta p_1 + (1-\eta)p_0$, $p_0:=1+4/d$, and $p_1:=1+4/(d-2)$.

In the case of $d=1,2$, we use
\begin{align*}
	\mathscr{V}:=L_{t,x}^{\frac{(d+2)(p-1)}{2}}, 
	\qquad \mathscr{U}:= \mathscr{H}_2 | \mathscr{D}_1
\end{align*}
instead of $\mathscr{V}_{\eta}$ and $\mathscr{Z}$.

\begin{proof}[Proof of Theorem \ref{thm1.4}]
First, we consider the case of $d\geq 3$. 
%In this case, we set 
%\begin{align*}
%	\mathscr{V}_{\eta}&:=[\mathscr{H}_2^0,\mathscr{H}_1]_{\eta} | [\mathscr{H}_2^0,\mathscr{W}_1]_{\eta}
%	\\
%	\mathscr{Z}&:=\mathscr{Z}_1 = \mathscr{H}_2 | \mathscr{W}_2 \cap \mathscr{D}_1
%\end{align*}
%where $p=\eta p_1 + (1-\eta)p_0$, $p_0:=1+4/d$, and $p_1:=1+4/(d-2)$.
Assume that $v$ is global and decay to $0$, that is, for any $\varepsilon>0$, there exists $T_{\varepsilon}>0$ such that 
\begin{align*}
	\|v(t)\|_{H^1} \leq \varepsilon
\end{align*}
for any $t \geq T_{\varepsilon}$. 

We set
\begin{align*}
	\mathcal{A}_{\tau}^1(t)(f_{\tau},g_{\tau})
	:=\left( \frac{1}{\tau} \mathcal{D}_{\tau}(t) + \partial_t \mathcal{D}_{\tau}(t) \right)f_{\tau} 
	+ \mathcal{D}_{\tau}(t) g_{\tau}
\end{align*}

By the Duhamel formula and Lemma \ref{lem3.7}, we have
\begin{align*}
	&\left\|u_{\tau}(t) -\mathcal{A}_{\tau}^1(t-T_{\varepsilon}) (u_{\tau}(T_{\varepsilon}), \partial_t u_{\tau}(T_{\varepsilon}) )\right\|_{\mathscr{Z} \cap \mathscr{V}_{\eta}(T_{\varepsilon},T)}
	\\
	&\lesssim \left\|\int_{T_{\varepsilon}}^{t} \frac{1}{\tau} \mathcal{D}_{\tau}(t-s) \mathcal{N}(u_{\tau}(s))ds\right\|_{\mathscr{Z} \cap \mathscr{V}_{\eta}(T_{\varepsilon},T)}
	\\
	&\lesssim \|u_{\tau}\|_{\mathscr{V}_{\eta}(T_{\varepsilon},T)}^{p-1} \|u_{\tau}\|_{\mathscr{Z}(T_{\varepsilon},T)},
\end{align*}
for $T>T_{\varepsilon}$. 
Now, we also have
\begin{align*}
	&\left\| \mathcal{A}_{\tau}^1(t-T_{\varepsilon}) (u_{\tau}(T_{\varepsilon}), \partial_t u_{\tau}(T_{\varepsilon}) ) \right\|_{\mathscr{Z} \cap \mathscr{V}_{\eta}(T_{\varepsilon},T)}
%	\\
%	&\lesssim \left\| \left( \frac{1}{\tau} \mathcal{D}_{\tau}(t) + \partial_t \mathcal{D}_{\tau}(t) \right)(u_{\tau}(T)-v(T)) \right\|_{\mathscr{Z} \cap \mathscr{V}_{\eta}(T,\infty)}
%	\\
%	&\quad+\left\| \left( \frac{1}{\tau} \mathcal{D}_{\tau}(t) + \partial_t \mathcal{D}_{\tau}(t) \right)v(T)\right\|_{\mathscr{Z} \cap \mathscr{V}_{\eta}(T,\infty)}
%	+\left\| \mathcal{D}_{\tau} \partial_t u_{\tau}(T) \right\|_{\mathscr{Z} \cap \mathscr{V}_{\eta}(T,\infty)}
	\\
	&\lesssim \|u_{\tau}(T_{\varepsilon})-v(T_{\varepsilon})\|_{H^1} + \|v(T_{\varepsilon})\|_{H^1} + \tau^{\frac{1}{2}} \| \partial_t u_{\tau}(T_{\varepsilon})\|_{L^2}.
\end{align*}
Therefore, if $\tau$ is sufficiently small, we obtain
\begin{align}
\label{eq3.1}
	\|u_{\tau}\|_{\mathscr{Z} \cap \mathscr{V}_{\eta}(T_{\varepsilon},T)}
	\lesssim \varepsilon +  \|u_{\tau}\|_{\mathscr{V}_{\eta}(T_{\varepsilon},T)}^{p-1} \|u_{\tau}\|_{\mathscr{Z}(T_{\varepsilon},T)}
\end{align}
since $\lim_{\tau \to 0} ( \|u_{\tau}(T_{\varepsilon})-v(T_{\varepsilon})\|_{H^1} + \tau^{\frac{1}{2}} \| \partial_t u_{\tau}(T_{\varepsilon})\|_{L^2})=0$. For any fixed small $\tau$, by taking $T=T(\tau)$ sufficiently close to $T_{\varepsilon}$, we obtain 
\begin{align}
\label{eq3.2}
	\|u_{\tau}\|_{\mathscr{Z} \cap \mathscr{V}_{\eta}(T_{\varepsilon},T)} \lesssim \varepsilon.
\end{align}
By \eqref{eq3.1}, \eqref{eq3.2}, and the bootstrap argument, we have
\begin{align}
\label{eq0}
	\|u_{\tau}\|_{\mathscr{Z} \cap \mathscr{V}_{\eta}(T_{\varepsilon},\infty)} \lesssim \varepsilon
\end{align}
for any small $\tau$. By the Strichartz estimate, we obtain
\begin{align*}
	&\|u_{\tau}\|_{L^\infty(T_{\varepsilon},\infty:H^1)} 
	\\
	&\lesssim \|u_{\tau}(T_{\varepsilon})-v(T_{\varepsilon})\|_{H^1} + \|v(T_{\varepsilon})\|_{H^1} + \tau^{\frac{1}{2}} \| \partial_t u_{\tau}(T_{\varepsilon})\|_{L^2} +  \|u_{\tau}\|_{\mathscr{Z} \cap \mathscr{V}_{\eta}(T_{\varepsilon},\infty)}^{p} 
	\\
	&\lesssim \varepsilon.
\end{align*}
This shows that
\begin{align*}
	\|u_{\tau} - v\|_{L^\infty(0,\infty:H^1)}
	&\leq \|u_{\tau} - v\|_{L^\infty(0,T_{\varepsilon}:H^1)} + \|u_{\tau} - v\|_{L^\infty(T_{\varepsilon},\infty:H^1)}
	\\
	&\lesssim   \|u_{\tau} - v\|_{L^\infty(0,T_{\varepsilon}:H^1)}  + \varepsilon
\end{align*}
for small $\tau$. Thus, we have
\begin{align*}
	\lim_{\tau \to 0} \|u_{\tau} - v\|_{L^\infty(0,\infty:H^1)}=0
\end{align*}
by Theorem \ref{thm1.3}. 

In the case of $d=1,2$, by the Strichartz estimates and the nonlinear estimate in Lemma \ref{lem3.7}, we have
\begin{align*}
	\left\| \mathcal{A}_{\tau}^1(t) (f,g) \right\|_{\mathscr{U} \cap \mathscr{V}_{\eta}}
	\lesssim \|f\|_{H^1} + \tau^{\frac{1}{2}} \| g\|_{L^2}.
\end{align*}
and 
\begin{align*}
%	&\left\|u_{\tau}(t) -\mathcal{A}_{\tau}(t-T) (u_{\tau}(T), \partial_t u_{\tau}(T) )\right\|_{\mathscr{Z} \cap \mathscr{V}(T,\infty)}
%	\\
%	&\lesssim 
	\left\|\int_{T}^{t} \frac{1}{\tau} \mathcal{D}_{\tau}(t-s) \mathcal{N}(u_{\tau}(s))ds\right\|_{\mathscr{U} \cap \mathscr{V}(T,\infty)}
	&\lesssim \|  \mathcal{N}(u_{\tau})\|_{\mathscr{R}_0 + \mathscr{R}_1(T,\infty)}
	\\
	&\lesssim \|u_{\tau}\|_{\mathscr{V}(T,\infty)}^{p-1} \|u_{\tau}\|_{\mathscr{U}(T,\infty)},
\end{align*}
since we have
\begin{align*}
	\mathscr{R}_0 = \mathscr{V}^{p-1} \mathscr{H}_2, 
	\qquad  \mathscr{R}_1 = \mathscr{V}^{p-1} \mathscr{D}_1.
\end{align*}
Therefore, we obtain the desired statement in the same way as in the higher dimensional case. 

We also have the statement for $\partial_t u_{\tau}$. Indeed, it holds from the Strichartz estimates that
\begin{align*}
	&\tau^{\frac{1}{2}} \|\partial_t u_{\tau}\|_{L_t^\infty L_x^2(T_{\varepsilon},\infty)} 
	\\
	&\lesssim \|u_{\tau}(T_{\varepsilon}) - v(T_{\varepsilon})\|_{H^1} + \|v(T_{\varepsilon})\| +\tau^{\frac{1}{2}}\|\partial_t u_{\tau}(T_{\varepsilon})\|_{L^2}
	+  \|u_{\tau}\|_{\mathscr{Z} \cap \mathscr{V}_{\eta}(T_{\varepsilon},\infty)}^{p} 
\end{align*}
if $d\geq 3$. 
When $d=1,2$, we have the similar estimate. Thus, we obtain the global convergence for $\partial_t u_{\tau}$.
\end{proof}

\begin{proof}[Proof of Thorem \ref{cor1.5}]
We show that $\lim_{t\to \infty} (t^{\frac{1}{2}}\| u_{\tau}(t)\|_{\dot{H}^1})=0$ uniformly in $\tau$. % Recall $\mathscr{Z}=\mathscr{Z}_1 = \mathscr{H}_2 | \mathscr{W}_2 \cap \mathscr{D}_1$. 

We consider the case of $d \geq 3$. Let $T>0$. 
Since $(t-T)^{\frac{1}{2}} \lesssim (t-s)^{\frac{1}{2}} + (s-T)^{\frac{1}{2}}$ for $t\geq s \geq T$, we have
\begin{align*}
	\|(t-T)^{\frac{1}{2}}u_{\tau}\|_{ \dot{\mathscr{Z}}(T,\infty)}
	&\lesssim  \| u_{\tau}(T)\|_{L^2} + \tau^{\frac{1}{2}} \| \partial_t u_{\tau}(T)\|_{L^2}
	\\
	&\quad +\left\|  \int_{T}^{t} \frac{1}{\tau}  (t-s)^{\frac{1}{2}} \mathcal{D}_{\tau}(t-s) \mathcal{N}(u_{\tau}(s))ds  \right\|_{ \dot{\mathscr{Z}}}
	\\
	&\quad +\left\|\int_{T}^{t} \frac{1}{\tau} \mathcal{D}_{\tau}(t-s) (s-T)^{\frac{1}{2}}\mathcal{N}(u_{\tau}(s))ds  \right\|_{ \dot{\mathscr{Z}}}
	\\
	&\lesssim  \varepsilon
	+\|\mathcal{N}(u_{\tau})\|_{\mathscr{R}_0^0 + \mathscr{R}_3^0(T,\infty)} 
	+\|(t-T)^{\frac{1}{2}}\mathcal{N}(u_{\tau})\|_{ \dot{\mathscr{R}}_0 + \dot{\mathscr{R}}_3(T,\infty)}
	\\
	&\lesssim \varepsilon +  \|u_{\tau}\|_{\mathscr{V}_{\eta}(T,\infty)}^{p-1} \|u_{\tau}\|_{\mathscr{Z}^0(T,\infty)} 
	\\
	& \quad +\|u_{\tau}\|_{\mathscr{V}_{\eta}(T,\infty)}^{p-1} \|(t-T)^{\frac{1}{2}}u_{\tau}\|_{ \dot{\mathscr{Z}}(T,\infty)}.
\end{align*}
Note that we use the fact that the spatial derivative implies the time decay $t^{-1/2}$ in the linear part of the first and second inequalities (see \cite{Inu19} for example). Combining \eqref{eq0} with the above inequality, we obtain
\begin{align*}
	\|(t-T)^{\frac{1}{2}}u_{\tau}\|_{ \dot{\mathscr{Z}}(T,\infty)} \lesssim  \varepsilon
\end{align*}
for large $T$ independent of $\tau$. 
By the Strichartz estimate, we have
\begin{align*}
	(t-T)^{\frac{1}{2}} \|u_{\tau}(t)\|_{\dot{H}^1}
	&\lesssim  \| u_{\tau}(T)\|_{L^2} + \tau^{\frac{1}{2}} \| \partial_t u_{\tau}(T)\|_{L^2}  +  \|u_{\tau}\|_{\mathscr{V}_{\eta}(T,\infty)}^{p-1} \|u_{\tau}\|_{\mathscr{Z}^0(T,\infty)} 
	\\
	& \quad +\|u_{\tau}\|_{\mathscr{V}_{\eta}(T,\infty)}^{p-1} \|(t-T)^{\frac{1}{2}}u_{\tau}\|_{ \dot{\mathscr{Z}}(T,\infty)}
	\\
	&\cleq \varepsilon.
\end{align*}
Thus, it holds that
\begin{align*}
	t^{\frac{1}{2}} \|u_{\tau}(t)\|_{\dot{H}^1} 
	&\lesssim T^{\frac{1}{2}} \|u_{\tau}(t)\|_{\dot{H}^1} + (t-T)^{\frac{1}{2}} \|u_{\tau}(t)\|_{\dot{H}^1}
	\\
	&\lesssim T^{\frac{1}{2}} \|u_{\tau}(t)\|_{\dot{H}^1} + \varepsilon
\end{align*}
for $t>T$. Since $\lim_{t\to \infty}\|u_{\tau}(t)\|_{\dot{H}^1} =0$ uniformly in $\tau$, we obtain the desired decay. %
%
%Since $\sup_{\tau>0}\| u_{\tau} \|_{\mathscr{E} \cap \mathscr{Z} \cap \mathscr{V}_{\eta}(0,\infty)}^{p-1} <1/2$ by taking small $\delta$, we obtain
%\begin{align*}
%	\|t^{\frac{1}{2}}u_{\tau}\|_{ \dot{\mathscr{Z}}(0,\infty)} 
%	\lesssim \delta +  \|u_{\tau}\|_{\mathscr{V}_{\eta}}^{p-1} \|u_{\tau}\|_{\mathscr{Z}^0} 
%	< C
%\end{align*}
%where $C$ is independent of $\tau$. 
The convergence for $\tau$ immediately follows from the same decay estimate of the solution $v$ to \eqref{NLH}. 

In the case of $d=1,2$, by using the function spaces $\mathscr{V}$ and $\mathscr{U}$ instead of $\mathscr{V}_{\eta}$ and $\mathscr{Z}$, respectively, we obtain the desired statement in the same way as above. 
The proof is completed. 
\end{proof}

%%%%%%%%%%%%%%%%%%%%%%%%%%%%%%%%%%%%%%%%%%%%%%%%%%%%%%%%

\appendix
\section{Some lemmas}
\label{appA}

\subsection{$L^\infty L^2$-$L^qL^r$ estimate and $L^qL^r$-$L^1 L^2$ estimate}

\begin{lemma}[$L^\infty L^2$-$L^qL^r$ estimate]

Let $\sigma\geq 0$, $2 \leq \tilde{r} < \infty$, and  $1\leq \tilde{q} \leq  \infty$. 
Assume that they satisfy 
\begin{align*}
	\frac{d}{2}\l( \frac{1}{2} - \frac{1}{\tilde{r}}\r)
	=\frac{1}{\tilde{q}},
\end{align*}
Then it holds that 
\begin{align*}
	\norm{ \jbra{\nabla}^{\sigma} \int_{0}^{t} \mathcal{D}_1(t-s) \chi_{\leq 1} F(s) ds}_{L^{\infty}(I:L^2(\R^d))}
	&\cleq \norm{F}_{L^{\tilde{q}'}(I:L^{\tilde{r}'}(\R^d))},
	\\
	\norm{\int_{0}^{t} \partial_t \mathcal{D}_1(t-s) \chi_{\leq 1} F(s) ds}_{L^{\infty}(I:L^2(\R^d))}
	&\cleq \norm{F}_{L^{\tilde{q}'}(I:L^{\tilde{r}'}(\R^d))},
\end{align*}
where $I=[0,T)$ % \subset [0,\infty)$ is a time interval such that $0 \in \overline{I}$ 
and the implicit constant is independent of $T$. 
\end{lemma}

\begin{proof}
%Let $P_{<1}=\mathcal{F}^{-1}\chi_{\leq 1} \mathcal{F}$.  
For simplicity, we set
\begin{align*}
	I(F)=I(F,\sigma):=\jbra{\nabla}^{\sigma} \int_{0}^{t} \mathcal{D}(t-s) \chi_{\leq 1} F(s) ds.
\end{align*}
We have
\begin{align*}
	\norm{P_j I(F)}_{L^2(\R^d)}
	&=\norm{ \int_{0}^{t}  P_j \jbra{\nabla}^{\sigma}\cD(t-s) \chi_{\leq 1}F(s) ds}_{L^2(\R^d)}
	\\
	&=\norm{ \int_{0}^{t}  P_j \jbra{\xi}^{\sigma} e^{-\frac{t-s}{2}}  L(t-s,\xi) \chi_{\leq 1} \hat{F}(s) ds}_{L_{\xi}^2(\R^d)}.
\end{align*}
If $j \geq 2$, then $P_j \chi_{\leq 1}=0$. The cases $j=-1,0,1$ are treated later. When $j \leq -2$, we have
\begin{align*}
	 P_j L(t-s,\xi)= P_j \frac{\sinh((t-s) \sqrt{1/4-|\xi|^2})}{\sqrt{1/4-|\xi|^2}}.
\end{align*}
Therefore, we have
\begin{align*}
	| P_j \jbra{\xi}^{\sigma} e^{-\frac{t-s}{2}}  L(t-s,\xi) \chi_{\leq 1}|
	&\cleq P_j e^{-\frac{t-s}{2}}|\sinh((t-s) \sqrt{1/4-|\xi|^2})| \chi_{\leq 1}
	\\
	&\cleq  P_j e^{-\frac{t-s}{2}} e^{(t-s)\sqrt{1/4-|\xi|^2}}\chi_{\leq 1}
	\\
	&\cleq P_j e^{-2(t-s)|\xi|^2}\chi_{\leq 1}
	\\
	&\cleq P_j e^{-2^{-1}(t-s)2^{2j}}\chi_{\leq 1}.
\end{align*}
By the Young inequality, we obtain
\begin{align*}
	&\norm{ \int_{0}^{t} P_j \jbra{\xi}^{\sigma} e^{-\frac{t-s}{2}}  L(t-s,\xi) \chi_{\leq 1} \hat{F}(s) ds}_{L_{\xi}^2(\R^d)}
	\\
	&\cleq \norm{ \int_{0}^{t} P_j e^{-2^{-1}(t-s)2^{2j}}\chi_{\leq 1}\hat{F}(s) ds}_{L_{\xi}^2(\R^d)}
	\\
	&\cleq \norm{e^{-2^{-1}t2^{2j}} }_{L^{\tilde{q}}} \norm{ P_j  \chi_{\leq 1}\hat{F}}_{L^{\tilde{q}'}L_{\xi}^2(\R^d)}
	\\
	&\cleq 2^{-2j/\tilde{q} } \norm{ P_j  \chi_{\leq 1}F}_{L^{\tilde{q}'} L^2(\R^d)},
\end{align*}
where $1\leq \tilde{q}\leq \infty$. 

In the case of $j=0,1$, we have
\begin{align*}
	 \left| e^{-\frac{t-s}{2}} P_j L(t-s,\xi) \right|= \left| e^{-\frac{t-s}{2}} P_j \frac{\sin((t-s) \sqrt{|\xi|^2-1/4})}{\sqrt{|\xi|^2-1/4}}\right| 
	 \cleq  e^{-\frac{t-s}{2}} (t-s) \in L^{\tilde{q}}_{s}(0,t)
\end{align*}
and thus we get 
\begin{align*}
	\norm{P_j I(F)}_{L^2(\R^d)} \cleq  \norm{ P_j  \chi_{\leq 1} F}_{L^{\tilde{q}'} L^2(\R^d)}
\end{align*}
for $j=0,1$. 

In the case of $j=-1$, we have
\begin{align*}
	 &|P_{-1}e^{-\frac{t-s}{2}}  L(t-s,\xi) \chi_{\leq 1}|
	 \\
	 &=
	 \begin{cases}
	\left|P_{-1} e^{-\frac{t-s}{2}} \frac{\sinh((t-s) \sqrt{1/4-|\xi|^2})}{\sqrt{1/4-|\xi|^2}} \chi_{\leq 1}\right|
	 \cleq e^{-\frac{t-s}{20}}\in L^{\tilde{q}}_{s}(0,t) & \text{ if } 1/4<|\xi|<1/2,
	 \\
	 \quad
	 \\
	  \left|P_{-1} e^{-\frac{t-s}{2}} \frac{\sin((t-s) \sqrt{|\xi|^2-1/4})}{\sqrt{|\xi|^2-1/4}} \chi_{\leq 1}\right| 
	 \cleq  e^{-\frac{t-s}{4}} \in L^{\tilde{q}}_{s}(0,t) & \text{ if } |\xi| >1/2.
	 \end{cases}
\end{align*}
 Thus, we get
\begin{align*}
	\norm{P_{-1} I(F)}_{L^2(\R^d)} \cleq  \norm{ P_{-1}  \chi_{\leq 1} F}_{L^{\tilde{q}'} L^2(\R^d)}
\end{align*}

We combine the above estimates. Since $L^2 \ceq \dot{B}_{2,2}^{0}$, it holds that
\begin{align*}
	\norm{I(F)}_{L^2(\R^d)}
	&\ceq \left( \sum_{j \in \Z} \norm{P_j I(F)}_{L^2(\R^d)}^2 \right)^{1/2}
	\\
	&\cleq \left( \sum_{j \in \Z} (2^{-2j/\tilde{q} } \norm{ P_j  \chi_{\leq 1} F}_{L^{\tilde{q}'} L^2(\R^d)} )^2 \right)^{1/2}
	\\
	&\ceq \norm{ \norm{  2^{-2j/\tilde{q} } \norm{P_j  \chi_{\leq 1} F}_{L^2(\R^d)} }_{L^{\tilde{q}'}}  }_{l_j^2}.
\end{align*}
Since $\tilde{q}' \leq 2$, by the Minkowskii integral inequality, we get
\begin{align*}
	\norm{ \norm{  2^{-2j/\tilde{q} } \norm{P_j  \chi_{\leq 1}F}_{L^2(\R^d)} }_{L^{\tilde{q}'}}  }_{l_j^2}
	&\cleq \norm{ \norm{  2^{-2j/\tilde{q} } \norm{P_j  \chi_{\leq 1}F}_{L^2(\R^d)} }_{l_j^2}  }_{L^{\tilde{q}'}}
	\\
	&\ceq \norm{ \norm{ \chi_{\leq 1} F }_{\dot{B}_{2,2}^{-2/\tilde{q}}}  }_{L_{t}^{\tilde{q}'}}
\end{align*}
The Sobolev inequality $\dot{B}_{2,2}^{-2/\tilde{q}} \supset \dot{B}_{\tilde{r},2}^{0}$,  where $\frac{d}{2}\l( \frac{1}{2} - \frac{1}{\tilde{r}}\r)=\frac{1}{\tilde{q}}$, and $\dot{B}_{\tilde{r}',2}^{0} \supset L^{\tilde{r}'}$ (since $\tilde{r}' \leq 2$) imply that 
\begin{align*}
	 \norm{ \norm{ \chi_{\leq 1}F }_{\dot{B}_{2,2}^{-2/\tilde{q}}}  }_{L_{t}^{\tilde{q}'}}
	\cleq  \norm{ \chi_{\leq 1}F }_{L_{t}^{\tilde{q}'}L_{x}^{\tilde{r}'}}. 
\end{align*}
\end{proof}

%$L^qL^r$-$L^\infty L^2$評価も使うので, それも証明しておかなくてはいけない.

\begin{lemma}[$L^qL^r$-$L^1 L^2$ estimate]
Let $2 \leq r <\infty$ and  $1\leq q \leq  \infty$. 
Assume that they satisfy 
\begin{align*}
	\frac{d}{2}\l( \frac{1}{2} - \frac{1}{r}\r)
	=\frac{1}{q},
\end{align*}
Then it holds that 
\begin{align*}
	\norm{  \int_{0}^{t} \mathcal{D}_1 (t-s) \chi_{\leq 1} F(s) ds}_{L^{q}(I:L^r(\R^d))}
	&\cleq \norm{F}_{L^{1}(I:L^{2}(\R^d))},
	\\
	\norm{\int_{0}^{t} \partial_t \mathcal{D}_1 (t-s) \chi_{\leq 1} F(s) ds}_{L^{q}(I:L^r(\R^d))}
	&\cleq \norm{F}_{L^{1}(I:L^{2}(\R^d))},
\end{align*}
%where $I \subset [0,\infty)$ is a time interval such that $0 \in \overline{I}$ and the implicit constant is independent of $I$. 
where $I=[0,T)$ % \subset [0,\infty)$ is a time interval such that $0 \in \overline{I}$ 
and the implicit constant is independent of $T$. 
\end{lemma}

\begin{proof}
This follows from $L^\infty L^2$-$L^qL^r$ estimate and the duality argument. 
%Before proving the statement, we prepare the following estimate.
%\begin{align*}
%	\norm{ \int_{s}^{T} \cD_{l}(t-s) F(t) dt}_{L_s^{\infty}((0,T):L^2(\R^d))}
%	\cleq \norm{F}_{L^{q'}((0,T):L^{r'}(\R^d))}
%\end{align*}
%This estimate follows from the similar argument to the above lemma. 
%Thus,  by the duality argument, we get
%\begin{align*}
%	\l| \int_{0}^{T} \int_{0}^{t} \tbra{\cD_{l}(t-s) f(s)}{g(t)}_{L^2} ds dt\r|
%	&=\l| \int_{0}^{T} \int_{0}^{t} \tbra{\cD_{l}(t-s) f(s)}{g(t)}_{L^2} ds dt\r| 
%	\\
%	&=\l| \int_{0}^{T} \int_{s}^{T} \tbra{f(s)}{\cD_{l}(t-s) g(t)}_{L^2}  dt ds\r| 
%	\\
%	&=\l| \int_{0}^{T}\tbra{f(s)}{ \int_{s}^{T}  \cD_{l}(t-s) g(t) dt}_{L^2}ds\r| 
%	\\
%	&\leq \norm{f}_{L^1(0,T:L^2)} \norm{ \int_{s}^{T}  \cD_{l}(t-s) g(t) dt}_{L_s^{\infty}(0,T:L^2)}
%	\\
%	&\cleq \norm{f}_{L^1(0,T:L^2)} \norm{g}_{L^{q'}((0,T):L^{r'}(\R^d))}
%\end{align*}
%Thus, we obtain
%\begin{align*}
%	\norm{ \int_{0}^{t}  \cD_{l}(t-s) f(s) ds}_{L^{q}((0,T):L^{r}(\R^d))}
%	\cleq  \norm{f}_{L^1(0,T:L^2)}.
%\end{align*}
\end{proof}

The Besov version can be also proved in the same way as in \cite{InWa19}. 
%%%%%%%%%%%%%%%%%%%%%%%%%%%%%%%%%%%%%%%%%%%%%%%%%%%%%%%%

\subsection{The Strichartz estimates for the high frequency part in the 1-d}
\label{appA.2}

We have the following Strichartz estimates for the high frequency part in the one dimensional case. 

\begin{lemma}[Homogeneous Strichartz estimates for the high frequency part in 1-d]
Let $d=1$.
Let $s\in \mathbb{R}$, $q \in [2,\infty]$, and $r\in [2,\infty]$. Then, we have the following.
\begin{align*}
	\norm{ \cD_{1}(t) \chi_{>1}  f}_{L_t^q \dot{B}^{s}_{r,2}(I)} 
	&\cleq  \norm{ |\nabla|^{\gamma(q,r)-1} \chi_{> 1} f}_{\dot{B}^{s}_{2,2}},
\end{align*}
and
\begin{align*}
	\norm{\partial_t \cD_{1}(t) \chi_{> 1}  f}_{L_t^q \dot{B}^{s}_{r,2}(I)} 
	\cleq  \norm{ |\nabla|^{\gamma(q,r)}  \chi_{>1} f}_{\dot{B}^{s}_{2,2}},
\end{align*}
where $\gamma(q,r)=1/2-1/r$. 
\end{lemma}

\begin{proof}
It is enough to consider 
\begin{align*}
	e^{-t/2} e^{\pm it\sqrt{-\Delta-1/4}} \chi_{>1}.
\end{align*}
As in \cite{Inu19}, we have
\begin{align*}
	\norm{e^{-t/2} e^{ it\sqrt{-\Delta-1/4}} \chi_{>1} P_j f}_{L^{q}(I:L^r(\R^d))}
	&= \norm{e^{-t/2} \norm{ e^{ it\sqrt{-\Delta-1/4}} \chi_{>1} P_j f }_{L^r}  }_{L^{q}(I)}
	\\
	&\cleq \norm{e^{-t/4} \norm{ e^{ it |\nabla|}  \chi_{>1} P_j  f }_{L^r}  }_{L^{q}(I)}
	\\
	&\cleq \norm{ e^{ it |\nabla|}  \chi_{>1} P_j f  }_{L^{\infty}(I:L^r(\R))},
\end{align*}
by the Mihlin--H\"{o}rmander multiplier theorem and the H\"{o}lder inequality. By the Bernstein inequality and the unitarity of the wave propagator, we obtain
\begin{align*}
	 \norm{ e^{ it |\nabla|} P_j f  }_{L^{\infty}(I:L^r(\R^d))} 
	& \lesssim  2^{j\left(\frac{1}{2}-\frac{1}{r}\right)} \norm{ e^{ it |\nabla|}  \chi_{>1} P_j f  }_{L^{\infty}(I:L^2(\R))}
	\\
	& \lesssim 2^{j\left(\frac{1}{2}-\frac{1}{r}\right)} \norm{ \chi_{>1} P_j f  }_{L^2}.
\end{align*}
Taking the $l^2$-norm for $j$, it holds that
\begin{align*}
	\norm{e^{-t/2} e^{ it\sqrt{-\Delta-1/4}} \chi_{>1} f}_{L^{q}(I:B^{0}_{r,2}(\R))}
	\lesssim \norm{|\nabla|^{\gamma}f}_{L^2}.
\end{align*}
This estimate implies the estimate for $\mathcal{D}_1$ and $\partial_t \mathcal{D}_1$. 
\end{proof}

\begin{lemma}[Inhomogeneous estimate for the high frequency in 1-d]
Let $d=1$. Let $s\in \mathbb{R}$, $q, \tilde{q} \in [2,\infty]$, and $r, \tilde{r}\in [2,\infty]$. 
Then, we have the following.
\begin{align*}
	\norm{ \int_0^t \cD_{1}(t-s) \chi_{> 1 } F(s) ds}_{L_t^q \dot{B}^{s}_{r,2}(I)} 
	\cleq
	 \norm{ |\nabla|^{\gamma(q,r)+\gamma(\tilde{q},\tilde{r})+\delta-1} \chi_{> 1 }F}_{L_t^{\tilde{q}'} \dot{B}^{s}_{\tilde{r}',2}(I)},
\end{align*}
where $\gamma(q,r)=1/2-1/r$. 
\end{lemma}

\begin{proof}
This follows from the argument in \cite{Inu19} as $d=1$.
\end{proof}

By scaling these estimate, we obtain Lemmas \ref{lem2.2.0} and \ref{lem2.4.0} in the case of $d=1$.

%%%%%%%%%%%%%%%%%%%%%%%%%%%%%%%%%%%%%%%%%%%%%%%%%%%%%%%%

%\section{Lemma only for the preprint}

\subsection{Completeness of $X|Y$}

We revisit the completeness of $X|Y$ for the reader's convenience. 
Let
\begin{align*}
	X|Y&:=\{ u : \| \chi_{\leq 1}(\nabla) u \|_{X} + \| \chi_{>1}(\nabla) u \|_{Y} <\infty\}
	\\
	\| u \|_{X|Y}&:=\| \chi_{\leq 1}(\nabla) u \|_{X} + \| \chi_{>1}(\nabla) u \|_{Y}
\end{align*}
where $X$ and $Y$ are Banach spaces satisfying $\| \chi_{\leq 1}(\nabla) u \|_{Z} \leq \| u \|_{Z}$ and  $\| \chi_{>1}(\nabla) u \|_{Z}\leq \| u \|_{Z}$ for $Z=X,Y$.

Take a Cauchy sequence $\{u_n\}$ in $X|Y$. 

As $m,n \to \infty$, we get
\begin{align*}
	&\|  \chi_{\leq 1}(\nabla)^2 u_n -  \chi_{\leq 1}(\nabla)^2 u_m \|_{X}
	\leq \|  \chi_{\leq 1}(\nabla) u_n -  \chi_{\leq 1}(\nabla) u_m \|_{X} \to 0
	\\
	&\|  \chi_{\leq 1}(\nabla) \chi_{>1}(\nabla)  u_n -  \chi_{\leq 1}(\nabla)  \chi_{>1}(\nabla)  u_m \|_{X}
	\leq \|  \chi_{\leq 1}(\nabla) u_n -  \chi_{\leq 1}(\nabla) u_m \|_{X} \to 0
\end{align*}
By the completeness of $X$, there exist $U^1$ and $U^2$ such that 
\begin{align}
\label{eq29.1}
	&\chi_{\leq 1}(\nabla)^2 u_n \to U^1 \text{ in } X
	\\
\label{eq29.2}
	&\chi_{\leq 1}(\nabla) \chi_{>1}(\nabla)  u_n \to U^2 \text{ in } X
\end{align}
We also have
\begin{align*}
	\|  \chi_{\leq 1}(\nabla) u_n -(U^1+U^2)\|_{X}
	&=\|  \chi_{\leq 1}(\nabla)( \chi_{\leq 1}(\nabla) u_n +  \chi_{> 1}(\nabla)u_n) -(U^1+U^2)\|_{X}
	\\
	&\leq \|  \chi_{\leq 1}(\nabla)^2 u_n -U^1\|_{X} +\| \chi_{\leq 1}(\nabla)\chi_{> 1}(\nabla)u_n -U^2\|_{X}
	\\
	& \to 0
\end{align*}
Therefore, we have
\begin{align}
\label{eq29.3}
	&\chi_{\leq 1}(\nabla) u_n \to U^1+U^2 \text{ in } X
\end{align}
Now, by \eqref{eq29.1}--\eqref{eq29.3}, we obtain
\begin{align}
\tag{1}
\label{eq1}
	\begin{cases}
	\chi_{\leq 1}(\nabla) (U^1 + U^2) = U^1
	\\
	\chi_{> 1}(\nabla) (U^1 + U^2) = U^2
	\end{cases}
\end{align}
The similar argument works in $Y$ for the high frequency part.
There exist $V^1,V^2$ such that 
\begin{align}
\label{eq29.4}
	&\chi_{> 1}(\nabla)^2 u_n \to V^1 \text{ in } Y
	\\
\label{eq29.5}
	&\chi_{\leq 1}(\nabla) \chi_{>1}(\nabla)  u_n \to V^2 \text{ in } Y
\end{align}
and
\begin{align}
\label{eq29.6}
	&\chi_{> 1}(\nabla) u_n \to V^1+V^2 \text{ in } Y
\end{align}
ans thus
\begin{align}
\tag{2}
\label{eq2}
	\begin{cases}
	\chi_{> 1}(\nabla) (V^1 + V^2) = V^1
	\\
	\chi_{\leq 1}(\nabla) (V^1 + V^2) = V^2
	\end{cases}
\end{align}

Noting \eqref{eq29.2} and \eqref{eq29.5}, $\chi_{\leq 1}(\nabla) \chi_{>1}(\nabla)  u_n$ converges in both $X$ and $Y$, and its limits in $X$ and $Y$ are $U^2$ and $V^2$, respectively. 

We get
\begin{align}
\tag{3}
\label{eq3}
	U^2=V^2.
\end{align}

\begin{proof}[Proof of {\eqref{eq3}}]
In general, let $f_n \to F$ in $X$ and $f_n \to G$ in $Y$. Then 
\begin{align*}
	\| F-G \|_{X+Y} 
	&=\| F-f_n +f_n-G \|_{X+Y}
	\\
	&\leq \| F-f_n \|_{X+Y} + \| f_n-G \|_{X+Y}
	\\
	&\leq \| F-f_n \|_{X} + \| f_n-G \|_{Y}
	\\
	& \to 0
\end{align*}
by the assumption. Therefore, $F=G$. 
\end{proof}

Set $u=U^1+U^2+V^1+V^2$. Then, we will show
\begin{align*}
	\| u_n - u \|_{X|Y} \to 0.
\end{align*}
Now, it follows from \eqref{eq1}, \eqref{eq2}, and \eqref{eq3} that
\begin{align*}
	\| \chi_{\leq 1}(\nabla) u_n - \chi_{\leq 1}(\nabla) u \|_{X}
	&=\| \chi_{\leq 1}(\nabla) u_n - \chi_{\leq 1}(\nabla) (U^1+U^2+V^1+V^2) \|_{X}
	\\
	&=\| \chi_{\leq 1}(\nabla) u_n - (U^1+V^2)\|_{X}
	\\
	&=\| \chi_{\leq 1}(\nabla) u_n - (U^1+U^2)\|_{X}
\end{align*}
By \eqref{eq29.3}, the last term goes to $0$ as $n \to \infty$. 
The same argument works in $Y$ for the high frequency term.
Thus, we get $\| u_n - u \|_{X|Y} \to 0.$

\acknowledgement

The first author is supported by JSPS KAKENHI Grant-in-Aid for Early-Career Scientists JP18K13444.

%%%%%%%%%%%%%%%%%%%%%%%%%%%%%%%%%%%%%%%%%%%%%%%%%%%%%%%%

%%%%%References%%%%%
%\bibliographystyle{plain}
%\bibliographystyle{abbrv}
%\bibliographystyle{myplain}
%\bibliography{references}

%
%\begin{thebibliography}{99}
%
%\bibitem{A}
%A,
%\textit{Templates}. 
%J. templates. {\bf 1}  (2014),  no. 1, 1-2. 
%
%\bibitem{B} 
%B,
%\textit{Templates}. 
%J. templates. {\bf 1}  (2016),  no. 2, 22-38. 
%
%\end{thebibliography}

\end{document}